\documentclass[dvips,preprint]{imsart}

\usepackage[OT1]{fontenc}

\usepackage{amsmath,amssymb,amsfonts,amsthm}
\usepackage[numbers]{natbib}
\usepackage{titlesec}
\usepackage[colorlinks,citecolor=blue,urlcolor=blue]{hyperref}
\usepackage[top=3cm, bottom=3cm, left=3cm, right=3cm]{geometry}
\usepackage{fancyhdr}
\usepackage{dsfont}
\usepackage{float}
\usepackage{url,graphicx,tabularx,array,appendix,geometry,multirow,array}
\usepackage{natbib}
\usepackage{trackchanges}
\usepackage{xy}\xyoption{all} \xyoption{poly} \xyoption{knot}
\usepackage{algorithmic}
\usepackage{algorithm}
\newcolumntype{L}[1]{>{\raggedright\let\newline\\\arraybackslash\hspace{0pt}}m{#1}}
\newcolumntype{C}[1]{>{\centering\let\newline\\\arraybackslash\hspace{0pt}}m{#1}}
\newcolumntype{R}[1]{>{\raggedleft\let\newline\\\arraybackslash\hspace{0pt}}m{#1}}

\renewcommand{\normalsize}{\fontsize{11pt}{\baselineskip}\selectfont}

\startlocaldefs
\numberwithin{equation}{section}
\theoremstyle{plain}
\newtheorem{theorem}{Theorem}[section]
\newtheorem{lemma}[theorem]{Lemma}

\newtheorem{corollary}[theorem]{Corollary}

\endlocaldefs

\newcommand{\bbR}{\mathbb{R}}

\begin{document}
\bibliographystyle{imsart-nameyear}
\begin{frontmatter}
\title{Sequential Markov Chain Monte Carlo}
\runtitle{Sequential Markov Chain Monte Carlo}

\begin{aug}
\author{\fnms{Yun} \snm{Yang}
\thanksref{t2}
\ead[label=e1]{yy84@stat.duke.edu}}
\and
\author{\fnms{David B.} \snm{Dunson}
\ead[label=e2]{dunson@stat.duke.edu}}

\thankstext{t2}{Supported by grant ES017436 from the National Institute of Environmental Health Sciences (NIEHS) of the National Institutes of Health
(NIH).}
\runauthor{Y. Yang et al.}

\affiliation{Duke University\thanksmark{m1}}

\address{Department of Statistical Science\\
Duke University\\
Box 90251\\
NC 27708-0251, Durham, USA\\
\printead{e1}\\
\phantom{E-mail:\ }\printead*{e2}}
\end{aug}

\begin{abstract}
We propose a sequential Markov chain Monte Carlo (SMCMC) algorithm to sample from a sequence of probability distributions, corresponding to posterior distributions at different times in on-line applications.  SMCMC proceeds as in usual MCMC but with the stationary distribution updated appropriately each time new data arrive.  SMCMC has advantages over sequential Monte Carlo (SMC) in avoiding particle degeneracy issues.  We provide theoretical guarantees for the marginal convergence of SMCMC under various settings, including parametric and nonparametric models.  The proposed approach is compared to competitors in a simulation study.  We also consider an application to on-line nonparametric regression.
\end{abstract}


\end{frontmatter}

\section{Introduction}
The Bayesian paradigm provides a natural formalism for optimal learning from data in a sequential manner, with the posterior distribution at one time point becoming the prior distribution at the next. Consider the following general setup. Let $\{\pi_t:t\in\mathcal{T}\}$ be a sequence of probability distributions indexed by discrete time $t\in\mathcal{T}=\{0,1,\ldots\}$. Assume that each $\pi_t$ can either be defined on a common measurable space $(E,\mathcal{E})$ or a sequence of measurable spaces $\{(E_t,\mathcal{E}_t):t\in\mathcal{T}\}$ with non-decreasing dimensions $d_0\leq d_1\leq\ldots$. Without loss of generality, we assume that $(E_t,\mathcal{E}_t)=(\bbR^{d_t},\mathcal{B}(\bbR^{d_t}))$, where $\mathcal{B}(\bbR^{d_t})$ is the Borel field on $\bbR^{d_t}$. Moreover, $\pi_t$ admits a density $\pi_t(\theta^{(t)})$ with respect to the Lebesgue measure $\lambda^{d_t}(d\theta^{(t)})$, where $\theta^{(t)}=(\theta^{(t-1)},\eta_t)$ is the quantity or parameter of interest at $t$ and $\eta_t\in\bbR^{d_t-d_{t-1}}$ is the additional component other than $\theta^{(t)}$. This framework can be considered as a generalization of \cite{Liu1998} from dynamic systems to arbitrary models or extension of \cite{Moral2006} from fixed space $E$ to time-dependent space $E_t$.

Many applications can be placed within this setting. In the sequential Bayesian inference context, $\theta^{(t)}$ corresponds to a vector composed of all the parameters and other unknowns to sample at time $t$. Similarly, $\pi_t$ is the posterior distribution of $\theta^{(t)}$ given the data collected until time $t$. For example, in generalized linear models with fixed number of covariates, $\theta^{(t)}$ includes the regression coefficients and residual variance and $d_t$ is a constant. In finite mixture models, $\theta^{(t)}$ includes both the parameters of the mixture components and mixing distribution, and the latent class indicators for each observation, so that $d_t$ is increasing with $t$. In state-space models, $\theta^{(t)}$ could be a vector composed of static parameters and state space variables, where the size of the latter grows with $t$. Even in batch situations where a full dataset $\{y_1,\ldots,y_n\}$ has been obtained, we can still consider the sequence of posterior distributions $p(\theta^{(t)}|y_1,\ldots,y_t)$ for $t\leq n$. The annealing effect of adding data sequentially can lead to substantial improvements over usual MCMC methods, which incorporate all the data at once and sample serially.

Markov Chain Monte Carlo (MCMC) is an important statistical analysis tool, which is designed to sample from complex distributions. It can not only be used for Bayesian analysis where a normalizing constant is unknown, but also for frequentist analysis when the likelihood involves high dimensional integrals such as in missing data problems and mixed effects models. However, in general, MCMC methods have several major drawbacks. First, it is difficult to assess whether a Markov chain has reached its stationary distribution. Second, a Markov chain can be easily trapped in local modes, which in turn would impede convergence diagnostics. To speed up explorations of the state space, annealing approaches introduce companion chains with flattened stationary distributions to facilitate the moves among separated high energy regions \citep{geyer1991,earlab2005,Kou2006}.

An alternative to MCMC is sequential Monte Carlo (SMC). The main idea of SMC is to represent the distribution $\pi_t$ through the empirical distribution $\hat{\pi}_t=\sum_{i=1}^NW_t^{(i)}\delta_{X_t^{(i)}}$, where $\{(W_t^{(i)},X_t^{(i)}):i=1,\ldots,N\}$ is a finite set of $N$ weighted particles with $\sum_{i=1}^NW_t^{(i)}=1$ and $\delta_x$ is the Dirac measure at $x$. As a new observation $y_{t+1}$ arrives, both weights and states of particles are updated in order to represent the new posterior $\pi_{t+1}$. Although SMC can potentially solve many of the drawbacks of MCMC mentioned above, it suffers from the notorious weight degeneracy issue where few particles quickly dominate as $t$ increases, causing performance based on $\hat{\pi}_t$ to degrade. Moreover, numerical errors introduced in an early stage can accumulate for some SMCs when static parameters are present \citep{Storvik2002}. Although many variants of SMC, such as adaptive importance sampling \citep{West1993}, resample-move strategies \citep{chopin2002} and annealed importance sampling \citep{Neal2001}, are proposed to alleviate the weight degeneracy problem, issues remain, particularly in models involving moderate to high-dimensional unknowns.

In this paper, we propose a sequential MCMC algorithm to sample from $\{\pi_t:t\in\mathcal{T}\}$ that is based on parallel sequential approximation algorithms. The proposed sequential MCMC is a population-based MCMC, where each chain is constructed via specifying a transition kernel $T_t$ for updating $\theta^{(t)}$ within time $t$ and a jumping kernel $J_t$ for generating additional component $\eta_t$. The annealing effect of sequential MCMC can substantially boost efficiency of MCMC algorithms with poor mixing rates with slight modifications. By exploiting multiple processors, SMCMC has comparable total computational burden as MCMC. For streaming data problems, SMCMC distributes this burden over time and allows one to extract current available information at any time point. We develop a rigorous theoretical justification on the convergence of SMCMC and provide explicit bounds on the error in terms of a number of critical quantities. The proofs also improve some existing results on the convergence of MCMC. The theory indicates an opposite phenomenon as the weight degeneracy effect of SMC: the deviations or numerical errors in the early stage decay exponentially fast as $t$ grows, leading to estimators with increasing accuracy.

The paper has the following organization. In Section 2, we present a generic SMCMC algorithm to sample from a sequence of distributions $\{\pi_t:t\in\mathcal{T}\}$ and discuss possible variations. In Section 3, we study the convergence properties of SMCMC under various settings, including parametric and nonparametric models. Section 4 compares SMCMC with other methods in a finite mixture of normals simulation. In Section 5, we apply SMCMC to an on-line nonparametric regression problem.

\section{Sequential Markov chain Monte Carlo}
We propose a sequential Markov chain Monte Carlo (SMCMC) class of algorithms in this section. The main idea of SMCMC is to run time-inhomogeneous Markov chains in parallel with the transition kernels depending on the current available data. Inferences can be made by using the ensemble composed of the last samples in those chains.

\subsection{Notation and assumptions}
Let $Y_t$ denote the data coming in at time $t$, $Y^{(t)}=(Y_1,\ldots,Y_t)$ the entire data up to $t$, $\theta^{(t)}$ the parameters at time $t$, $d_t$ the size of $\theta^{(t)}$ and $\pi^{(t)}(\theta^{(t)})$ the prior distribution, implying that we can add parameters over time. Although not necessary, for notational simplicity we assume that the prior is compatible: $\pi^{(t)}(\theta^{(t)})=\int\pi^{(t+1)}(\theta^{(t)},\eta_{t+1})d\eta_{t+1}$ with $\theta^{(t+1)}=(\theta^{(t)},\eta_{t+1})$. Under this assumption, we can suppress the superscript $t$ in $\pi^{(t)}$. The compatibility assumption is a consequence of the restriction that if the extra parameters in the prior at time $t+1$ are marginalized out, then we recover the prior at time $t$. This restriction is trivially satisfied under the special case when $d_t$ does not grow with time, and is also true under more general priors such as hierarchical priors for mixed effects models and Gaussian process priors for nonparametric regression. We propose to conduct $L$ Markov chains in parallel exploiting $L$ processors to obtain samples, $\theta^{(t,l)}=\{\theta^{(1,t,l)},\ldots,\theta^{(m_t,t,l)}\}$ for $t=1,2,\ldots$ and $l=1,\ldots,L$, where $m_t$ is the number of draws obtained at time $t$ for each chain and $\theta^{(s,t,l)}\in\bbR^{d_t}$ is the $s$th draw obtained in the $l$th chain at $t$. The ensemble $\Theta_t=\{\theta^{m_t,t,l}:l=1,\ldots,L\}$ will be treated as independent draws sampled from the posterior $\pi_t(\theta^{(t)})=\pi(\theta^{(t)}|Y^{(t)})$ at time $t$.

\subsection{Markov chain construction}\label{se:al}
At each time $t$, we consider two kernels: a jumping kernel $J_t$ proposing the parameter jumping from $t-1$ to $t$ at the beginning of time $t$ and a transition kernel $T_t$ specifying the parameter updating process within time $t$. $J_t(\cdot,\cdot)$ is defined on $\bbR^{d_{t-1}}\times\bbR^{d_{t}}$ and is primarily designed for the situation when the parameter grows at $t$. In the case when $d_t=d_{t-1}$, $J_t$ could be chosen as the identity map.
$T_t(\cdot,\cdot)$ is defined on $\bbR^{d_t}\times\bbR^{d_t}$ so that the posterior $\pi_t$ is the stationary measure of the Markov chain with transition kernel $T_t$, i.e.
\[
\pi_t(\theta')=\int_{\bbR^{d_t}}\pi_t(\theta)\ T_t(\theta,\theta')d\theta.
\]
$T_t$ aims at transferring the distribution of the draws in $\Theta_{t-1}$ from $\pi_{t-1}$ to $\pi_{t}$. From standard Markov chain theory \citep{Meyn1993}, if the chain with transition kernel $T_t$ is an aperiodic recurrent Harris chain, then $||T_t^{m_t}\circ p_0-\pi_{t}||_1\to 0$ as $m_t\to\infty$ for any initial distribution $p_0$. Therefore, as we repeat applying the transition $T_t$ for enough times, the distribution of $\Theta_t$ will converge to $\pi_{t}$. Theorem \ref{prop:p1} in section \ref{se:cmt} quantifies such approximation error with given $m_t$. Section \ref{se:mt} provides recommendations on automatically choosing $m_t$ in practice.

We construct our SMCMC based on $J_t$ and $T_t$ as follows:
\begin{enumerate}
  \item At $t=0$, we set $m_t=1$ and draw $L$ samples from a known distribution, for example, the prior $\pi=\pi_0$. The samples at $t=0$ are denoted as $\theta^{(1,0,1)},\ldots,\theta^{(1,0,L)}$.
  \item At $t>0$, we first update $\theta^{(m_{t-1},t-1,l)}$ to $\theta^{(1,t,l)}$ through the jumping kernel $J_t$ as
\[
P\big(\theta^{(1,t,l)}\big|\theta^{(m_{t-1},t-1,l)}\big)
=J_t\big(\theta^{(m_{t-1},t-1,l)},\theta^{(1,t,l)}\big),
\]
in parallel for $l=1,\ldots,L$. Then, for $s=1,\ldots,m_{t}-1$, $\theta^{(s,t,l)}$ is sequentially transited to $\theta^{(s+1,t,l)}$ through the transition kernel $T_t$ as
\[
P\big(\theta^{(s+1,t,l)}\big|\theta^{(s,t,l)}\big)
=T_t\big(\theta^{(s,t,l)},\theta^{(s+1,t,l)}\big),
\]
in parallel for $l=1,\ldots,L$.
\end{enumerate}
With the above updating scheme, the last samples $\{\theta^{(m_t,t,l)}:l=1,\ldots,L\}$ at $t$ would be taken as the ensemble $\Theta_t$ to approximate the posterior $\pi_t$. Theorem \ref{thm:main} in section \ref{se:cmt} and Theorem \ref{thm:mainb} in section \ref{se:ipd} guarantee the error decays to zero as $t$ increases to infinity as long as $||\pi_t-\pi_{t-1}||_1\to 0$. When $d_t$ is growing, the $\pi_t$ in the $L_1$ norm is the marginal distribution of $\theta^{(t-1)}$ given by
\begin{align}\label{eq:mar}
\pi_t(\theta^{(t-1)})=\int_{\bbR^{d_t-d_{t-1}}}\pi_t(\theta^{(t-1)},\eta_t)d\eta_t.
\end{align}
The sequential Monte Carlo sampler \citep{Moral2006} could also be cast into this framework if the jumping kernel $J_t$ is a random kernel that depends on $\Theta_{t-1}$. However, as Theorem \ref{thm:main} indicates, with sufficient iterations $m_t$ at each time point $t$, one can guarantee the convergence without the resampling step used in SMC algorithms as long as the posterior $\pi_t$ does not change too much in $t$.

As the mixture model example in section \ref{se:mixm} demonstrates, even in batch problems, the annealing effect of adding data sequentially will lead to substantial improvements over usual MCMC algorithms that incorporate all the data at once and sample serially. For streaming data problems, SMCMC avoids the need to restart the algorithm at each time point as new data arrive, and allows real time updating exploiting multiple processors and distributing the computational burden over time. For example, the SMCMC for nonparametric probit regression in section \ref{se:npr} has similar total computational burden as running MCMC chains in parallel using multiple processors. However, SMCMC distributes this burden over time, and one can extract current available information at any time point.  Moreover, the samples $\{\theta^{(m_t,t,l)}:l=1,\ldots,L\}$ within each time point are drawn from independent chains. This independence and the annealing effect can substantially boost efficiency of MCMC algorithms with poor mixing rates.

\subsection{Choice of $J_t$}\label{se:Jt}
We shall restrict the jumping kernel $J_t$ to be a pre-specified transition kernel that leaves $\theta^{(t-1)}$ unchanged by letting
\begin{align}\label{eq:cJt}
    P\big((\tilde{\theta}^{(t-1)},\eta_t)|\theta^{(t-1)}\big)=
    J_t\big(\theta^{(t-1)},(\tilde{\theta}^{(t-1)},\eta_t)\big)
    \delta_{\theta^{(t-1)}}(\tilde{\theta}^{(t-1)}),
\end{align}
where $\delta_x$ is the Dirac measure at $x$.
Otherwise, $J_t$ can always be decomposed into an updating of $\theta^{t-1}$ followed by a generation of $\eta_t$, where the former step can be absorbed into $T_{t-1}$. Henceforth, with slight abuse of notation, the jumping kernel $J_t$ will be considered as a map from $\bbR^{d_{t-1}}$ to $\bbR^{d_t-d_{t-1}}$, mapping $\theta^{(t-1)}$ to $\eta_t$.

Intuitively, if $\theta^{(t-1)}$ is approximately distributed as $\pi_t(\theta^{(t-1)})$ and $\eta_t$ is sampled from the conditional posterior $\pi_t(\eta_t|\theta^{(t-1)})$, then $(\theta^{(t-1)},\eta_t)$ is approximately distributed as
\[
\pi_{t}(\theta^{(t-1)},\eta_t)=\pi_t(\theta^{(t-1)})\pi_t(\eta_t|\theta^{(t-1)}),
 \]
the exact posterior distribution.  This observation is formalized in Lemma \ref{le:Jt} in section \ref{se:ipd}, suggesting that the jumping kernel $J_t$ should be chosen close to full conditional $\pi_t(\eta_t|\theta^{(t-1)})$ at time $t$. Two types of $J_t$ can be used:
\begin{enumerate}
  \item \emph{Exact conditional sampling.} When draws from the full conditional $\pi_t(\eta_t|\theta^{(t-1)})$ can be easily sampled, $J_t$ can be chosen as this full conditional. For example, $\pi_t(\eta_t|\theta^{(t-1)})$ can be recognized as some standard distribution. Even when $\pi_t(\eta_t|\theta^{(t-1)})$ is unrecognizable, if $d_t-d_{t-1}$ is small, then we can apply the accept-reject algorithm \citep{robert2004} or slice sampler \citep{Neal2003}.
  \item \emph{Approximate conditional sampling.} When sampling from the full conditional of $\eta_t$ is difficult, we can use other transition kernels, such as blocked Metropolis-Hastings (MH) or inter-woven MH or Gibbs steps chosen to have $\pi_t(\eta_t|\theta^{(t-1)})$ as the stationary distribution.
\end{enumerate}

Theorem \ref{thm:mainb} in section \ref{se:ipd} provides an explicit expression about the impact of
\[
\lambda_t=\sup_{\theta^{(t-1)}\in\bbR^{d_{t-1}}}||\pi_t
(\cdot|\theta^{(t-1)})-J_t(\theta^{(t-1)},\cdot)||_1
\]
on the approximation error of $\pi_t$, which basically requires $\lambda_t\to 0$ as $t\to \infty$. To achieve $\lambda_t\to 0$, one can run the transition kernel in approximate conditional sampling case for an increasing number of iterations as $t$ grows. However, we observe good practical performances for a fixed small number of iterations.

\subsection{Choice of $T_t$}\label{se:Tt}
Lemma \ref{le:le2} in section \ref{se:pre} suggests that a good $T_t(\theta,\theta')$ should be close to $\pi_t(\theta')$. The transition kernel $T_t$ can be chosen as in usual MCMC algorithms.  For example, $T_t$ can be the transition kernel associated with blocked or inter-weaved MH or Gibbs samplers.  For conditionally conjugate models, it is particularly convenient to use Gibbs and keep track of conditional sufficient statistics to mitigate the increase in storage and computational burden over time.

\subsection{Choice of $m_t$}\label{se:mt}
The number of samples in each chain per time point, $m_t$, should be chosen to be small enough to meet the computational budget while being large enough so that the difference between the distribution of samples in $\Theta_t$ and the posterior distribution $\pi_t$ goes to zero. Formal definitions of difference and other concepts will be given in the next section. Intuitively, for a given $t$, if the Markov chain with transition kernel $T_t$ has slow mixing or there are big changes in $\pi_t$ from $\pi_{t-1}$, then $m_t$ should be large. Theorem \ref{thm:main} in section \ref{se:cmt} provides explicit bounds on the approximation error as a function of $m_t$'s. Moreover, for a given $\epsilon\in(0,1)$, Theorem \ref{thm:main} implies that if we select $m_t$ to be the minimal integer $k$ such that
\begin{align*}
    r_t(k)\leq 1-\epsilon,
\end{align*}
where $r_t$ is the rate function associated with $T_t$ defined in \eqref{eq:ue},
then the distribution of $\Theta_t$ converges to $\pi_t$ as $t\to\infty$ under the assumption that $||\pi_t-\pi_{t-1}||_1\to 0$. Typical rate functions can be chosen as $r_t(k)=\rho^k$, for some $\rho^k$.
Since the rate functions $r_t$ relate to the unknown mixing rate of the Markov chain with transition kernel $T_t$, we estimate them in an online manner.

To estimate $r_t$ we utilize the relationship between the mixing rate of a Markov chain and its autocorrelation function.  By comparing \eqref{eq:nor} and \eqref{eq:corr} in section \ref{se:semt}, the decay rate of the autocorrelation function provides an upper bound for the mixing rate. Therefore, we can bound the rate function $r_t(k)$ with the lag-$k$ autocorrelation function
\begin{align*}
    f_t(k)=\max_{j=1,\ldots,p} \text{ corr} (X_{j}^{(k)},X_j^{(0)}),
\end{align*}
where $(X_j^{(1)},\ldots,X_j^{(p)})$ is the $p$-dimensional sample in the $k$th step of the Markov chain with transition kernel $T_t$.

For a single Markov chain, the common choice of estimating $f_t(k)$ by the sample average of lag-$k$ differences over the steps from $s=s_1,\ldots,s_2$ as
\begin{align*}
    \tilde{f}_t(k)=\max_{j=1,\ldots,p}\frac{\sum_{s=s_1}^{s_2}(X_j^{(s)}-\bar{X}_j)
    (X_j^{(s-k)}-\bar{X}_j)}
    {\sum_{s=s_1}^{s_2}(X_j^{(s)}-\bar{X}_j)^2},
\end{align*}
where $\bar{X}_j=\sum_{s=s_1}^{s_2}X_j^{(s)}/(s_2-s_1+1)$, could have large bias even though $s_2-s_1$ is large. The reason is that for slow mixing Markov chains, the samples tend to be stuck in local modes, leading to high variation of $\tilde{f}_t(k)$'s with $X_j^{(s)}$ starting from different regions. Within these local modes, $\tilde{f}_t(k)$ might decay fast, inappropriately suggesting good mixing.
In our algorithm, we have $L$ chains running independently in parallel. Hence, instead of averaging over time, we can estimate the autocorrelation function $f_t(k)$ by averaging across the independent chains as
\begin{align*}
    \hat{f}_t(k)=\max_{j=1,\ldots,p}\frac{\sum_{l=1}^L(X_j^{(k,l)}-\bar{X}_j^{(k)})
    (X_j^{(0,l)}-\bar{X}_j^{(0)})}
    {\big(\sum_{l=1}^L(X_j^{(k,l)}-\bar{X}_j^{(k)})^2\big)^{1/2}
    \big(\sum_{l=1}^L(X_j^{(0,l)}-\bar{X}_j^{(0)})^2\big)^{1/2}},
\end{align*}
where $X_j^{(k,l)}$ is the $j$th component of the sample in the $k$th step of the $l$th chain and $\bar{X}_j^{(k)}=\sum_{l=1}^{L}X_j^{(k,l)}/L$ is the ensemble average of the draws in the $k$th step across the $L$ Markov chains. $\hat{f}_t$ will be more robust than $\tilde{f}_t$ to local modes. Although by Slutsky's theorem, both estimators are asymptotically unbiased as $s_2-s_1\to\infty$ and $L\to\infty$ respectively, the convergence of $\tilde{f}_t$ might be much slower than that of $\hat{f}_t$ due to potential high correlations among the summands in $\tilde{f}_t$.

In our case, the estimator $\hat{f}_t(k)$ takes the form of
\begin{align}\label{eq:ftk}
    \hat{f}_t(k)=\max_{j=1,\ldots,p}\frac{\sum_{l=1}^L(\theta_j^{(k+1,t,l)}-\bar{\theta}_j^{(k+1,t)})
    (\theta_j^{(1,t,l)}-\bar{\theta}_j^{(1,t)})}
    {\big(\sum_{l=1}^L(\theta_j^{(k+1,t,l)}-\bar{\theta}_j^{(k+1,t)})^2\big)^{1/2}
    \big(\sum_{l=1}^L(\theta_j^{(1,t,l)}-\bar{\theta}_j^{(1,t)})^2\big)^{1/2}},
\end{align}
where $\bar{\theta}_j^{(k,t)}=\sum_{l=1}^{L}\theta_j^{(k,t,l)}/L$ is the $j$th component of the ensemble average of the draws across the $L$ Markov chains in the $k$th step at time $t$. For each $t>0$, we choose $m_t$ to be the minimal integer $k$ such that the sample autocorrelation decreases below $1-\epsilon$, i.e.
\begin{align*}
    m_t=\min\{k:\hat{f}_t(k)\leq 1-\epsilon\}.
\end{align*}
In practice, we can choose $\epsilon$ according to the full sample size $n$ and error tolerance $\epsilon_T$ based on Theorem \ref{thm:main}. For example, for small datasets with $n\sim 10^2$, we recommend $\epsilon=0.5$ and for large datasets, $\epsilon$ such that
\begin{align*}
    \sum_{t=1}^{n}\frac{\epsilon^{n+1-t}}{\sqrt{t}}\leq\epsilon_T,
\end{align*}
where $t^{-1/2}$ is a typical rate for $||\pi_t-\pi_{t-1}||_1$
for regular parametric models (Lemma \ref{le:le3}).
To summarize, Algorithm \ref{al:1} provides pseudo code for SMCMC.
\begin{algorithm*}
\caption{Sequential Markov Chain Monte Carlo}
\label{al:1}
\begin{algorithmic}
 \STATE  $m_0\leftarrow1$
 \FOR{$l=1$ \TO $L$ } \STATE{Draw $\theta^{(1,0,l)}\sim \pi_0$} \ENDFOR
 \FOR{$t=1$ \TO $n$ }
 \STATE  $m_t\leftarrow1$
 \STATE  $\rho\leftarrow1$
 \FOR{$l=1$ \TO $L$ } \STATE{Draw $[\eta^{(t,l)}\ |\ \theta^{(m_{t-1},t-1,l)}]\sim J_t(\theta^{(m_{t-1},t-1,l)},\cdot)$}
 \STATE{$\theta^{(1,t,l)}\leftarrow(\theta^{(m_{t-1},t-1,l)},\eta^{(t,l)})$} \ENDFOR
 \WHILE{$\rho>1-\epsilon$}
 \STATE{$m_t\leftarrow m_t+1$}
 \FOR{$l=1$ \TO $L$ }
 \STATE{ Draw $[\theta^{(m_t,t,l)}\ |\ \theta^{(m_t-1,t,l)}]\sim T_t(\theta^{(m_t-1,t,l)},\cdot)$}
 \ENDFOR
 \STATE{ Calculate $\hat{f}_t(m_t-1)$ by \eqref{eq:ftk}}
 \STATE{ $\rho\leftarrow\hat{f}_t(m_t-1)$}
 \ENDWHILE
 \STATE{$\Theta_t\leftarrow\{\theta^{(m_t,t,l)}:l=1,\ldots,L\}$}
 \ENDFOR
\end{algorithmic}
\end{algorithm*}

All the loops for $l$ in the above algorithm can be computed in parallel. Assuming the availability of a distributed computing platform with multiple processors, Algorithm 1 has comparable computational complexity to running MCMC in parallel on $L$ processors starting with the full data at time $t$.  The only distributed operation is computation of $\hat{f}_t$, which can be updated every $s_0$ iterations to reduce communication time.  Moreover, the $t$ loop can be conducted whenever $t_0$ ($>1$) new data points accrue, rather than as each data point arrives, as long as $||\pi_t-\pi_{t-t_0}||\to 0$ as $t\to\infty$. More generally, for any sequence $t_1<t_2<\ldots<t_{k_0}=n$ such that $||\pi_{t_k}-\pi_{t_{k-1}}||\to 0$ as $k\to\infty$, the loop for $t$ can be changed into ``for $k=1$ to $k_0$ do $t\leftarrow t_k\ \ldots$ end for''. Since the posterior $\pi_t$ is expected to vary slower as $t$ grows, the batch sizes $t_{k}-t_{k-1}$ can be increasing in $k$, leading to faster computations. To avoid the SMCMC becoming too complicated, we shall restrict our attention to Algorithm \ref{al:1} in the rest of the paper.

\section{Convergence results}
In this section, we study the convergence properties of SMCMC as $t\to \infty$. We first review some convergence results for Markov chains and then prove some new properties. We then apply these tools to study SMCMC convergence.

\subsection{Convergence of Markov chain}\label{se:pre}
Denote the state space by $\mathcal{X}$ and the Borel $\sigma$-algebra on $\mathcal{X}$ by $\mathcal{B(\mathcal{X})}$. For a transition kernel $T(x,y)$, we can recursively define its $t$-step transition kernel by
\[
T^t(x,y)=\int T^{t-1}(x,z)T(z,y)dz.
\]
Similarly, given an initial density $p_0$, we will denote by $T^t\circ p_0$ the probability measure evolved after $t$th steps with transition kernel $T$ from the initial distribution $p_0$, which can be related to $T^t$ by
\begin{align*}
    T^t\circ p_0(x)=\int T^t(z,x)p_0(z)dz.
\end{align*}

A transition kernel $T$ is called uniformly ergodic if there exists a distribution $\pi$ and a sequence $r(t)\rightarrow 0$, such that for all $x$,
\begin{align}\label{eq:ue}
    ||T^t(x,\cdot)-\pi||_1\leq r(t),
\end{align}
where $||\cdot||_1$ is the $L_1$ norm. $r(t)$ will be called the rate function. If $T$ is ergodic, then $\pi$ in the definition will be the stationary distribution associated with $T$. Uniformly ergodic implies geometric convergence, where $r(t)=\rho^t$ for some $\rho\in(0,1)$ \citep{Meyn1993}.

We call a transition kernel $T$ universally ergodic if there exists a distribution $\pi$ and a sequence $r(t)\rightarrow 0$, such that for any initial distribution $p_0$,
\begin{align*}
    ||T^t\circ p_0-\pi||_1\leq r(t)||p_0-\pi||_1.
\end{align*}
$r(t)$ will also be called rate function. The concept of universal ergodicity plays an important role in the following study of the convergence properties of SMCMC. By choosing $p_0$ as a Dirac measure at $x$, one can see that universal ergodicity implies uniform ergodicity with rate function $2r(t)$. In addition, universal ergodicity can provide tighter bounds on the MCMC convergence than uniform ergodicity especially when the initial distribution $p_0$ is already close to $\pi$. The following lemma provides the converse. The proof is based on coupling techniques.

\begin{lemma}\label{le:le1}
If a transition kernel $T$ is uniformly ergodic with rate function $r(t)$, then it is universally ergodic with the same rate function.
\end{lemma}

\begin{proof}
We will use coupling method. Denote $\delta=\frac{1}{2}||p_0-\pi||_1$. Let $\{X_t:t\geq0\}$ and $\{X'_t:t\geq0\}$ be two Markov chains defined as follows:
\begin{enumerate}
  \item $X_0\sim p_0$;
  \item Given $X_0=x$, with probability $\min\{1,\frac{\pi(x)}{p_0(x)}\}$, set $X'_0=x$;
      with probability $1-\min\{1,\frac{\pi(x)}{p_0(x)}\}$, draw
      \[
      X'_0\sim\frac{\pi(\cdot)-\min\{\pi(\cdot),p_0(\cdot)\}}{\delta};
      \]
  \item For $t\geq1$, if $X_0=X_0'$, draw $X_t=X'_t\sim T(X_{t-1},\cdot)$, else draw $X_t$ and $X'_t$ independently from $X_t\sim T(X_{t-1},\cdot)$ and $X'_t\sim T(X'_{t-1},\cdot)$ respectively.
\end{enumerate}
Note that $\frac{\pi(\cdot)-\min\{\pi(\cdot),p_0(\cdot)\}}{\delta}$ is a valid probability density since: 1. it is nonnegative; 2. its integral on $\mathcal{X}$ is equal to one by the definition of $\delta$.

From the above construction, it is easy to see that the marginal distribution of $X_t$ is $T^t\circ p_0$. Next we will prove that the marginal distribution of $X'_t$ is $\pi$ for all $t$. Since the stationary distribution of $T$ is $\pi$, we only need to show that the marginal distribution of $X'_0$ is $\pi$. First,
\begin{equation}\label{eq:diff}
\begin{aligned}
    P(X_0= X'_0)=&\int \min\{1,\frac{\pi(x)}{p_0(x)}\} p_0(x)dx\\
    =&\int \min\{p_0(x),\pi(x)\}dx\\
    =&1-\delta.
\end{aligned}
\end{equation}
Then, for any $A\in\mathcal{B(\mathcal{X})}$,
\begin{align*}
    P(X'_0\in A)=&P(X'_0\in A,X_0\neq X'_0)+P(X'_0\in A,X_0= X'_0)\\
    =&P(X'_0\in A|X_0\neq X'_0)P(X_0\neq X'_0)+\int_AP(X_0=X'_0|X'_0=x)P(X'_0=x)dx\\
    =&\delta \int_A \frac{\pi(x)-\min\{\pi(x),p_0(x)\}}{\delta}dx\
    +\int_A\min\bigg\{1,\frac{\pi(x)}{p_0(x)}\bigg\}p_0(x)dx\\
    =&\int_A\pi(x)dx.
\end{align*}

By uniform ergodicity, for any probability measure $p$, we have
\begin{equation}\label{eq:total}
\begin{aligned}
    ||T^t\circ p-\pi||_1=&\int\bigg|\int T^t(z,x)p(z)dz-\int \pi(x) p(z)dz\bigg|dx\\
    \leq &\int||T^t(z,\cdot)- \pi(\cdot)||_1p(z)dz\\
    \leq & r(t).
\end{aligned}
\end{equation}

By the above inequality, \eqref{eq:diff} and our construction of $X_t$ and $X'_t$, for any $A\in\mathcal{B(\mathcal{X})}$, we have
\begin{align*}
    |T^t\circ p_0(A)-\pi(A)|=&|P(X_t\in A)-P(X'_t\in A)|\\
    = & |P(X_0\neq X'_0, X_t\in A)-P(X_0\neq X'_0, X'_t\in A)|\\
    \leq & P(X_0\neq X'_0)\ \big\{|P(X_t\in A|X_0\neq X'_0)-\pi(A)|\\
    &+|P(X'_t\in A|X_0\neq X'_0)-\pi(A)|\big\}\\
    \leq & \delta r(t),
\end{align*}
where the last line follows by the fact that $||p-q||_1=2\sup_{A}|p(A)-q(A)|$ and \eqref{eq:total} with $p(\cdot)=P(X_0=\cdot|X_0\neq X'_0)$ and $p(\cdot)=P(X'_0=\cdot|X_0\neq X'_0)$. Therefore,
\begin{align*}
    ||T^t\circ p_0-\pi||_1=2\sup_A |T^t\circ p_0(A)-\pi(A)|\leq r(t)||p_0-\pi||_1.
\end{align*}
\end{proof}

The coupling in the proof of Lemma \ref{le:le1} is constructed through importance weights. By using the same technique, we can prove the uniform ergodicity for certain $T$ as in the following lemma.

\begin{lemma}\label{le:le2}
If the transition kernel $T$ satisfies
\begin{align}
    \sup_x||T(x,\cdot)-\pi||_1\leq 2\rho, \label{eq:ul1}
\end{align}
for some $\rho<1$, then $T$ is uniformly ergodic with rate function $r(t)=\rho^t$.
\end{lemma}

\begin{proof}
Let $\delta(x)=\frac{1}{2}||T(x,\cdot)-\pi||_1\leq \rho$.
Given an initial point $x$, we can construct two Markov chains $\{X_t:t\geq0\}$ and $\{X'_t:t\geq0\}$ as follows:
\begin{enumerate}
  \item $X_0=x$, $X'_0\sim \pi$;
  \item For $t\geq1$, given $X_{t-1}=x$ and $X'_{t-1}=x'$,
     \begin{enumerate}
       \item[(a)] if $x=x'$, choose $X_t=X'_t\sim T(x,\cdot)$;
       \item[(b)] else, first choose $X'_t=y\sim T(x',\cdot)$,
       then with probability $\min\{1,\frac{T(x,y)}{\pi(y)}\}$, set $X_t=y$,
       with probability $1-\min\{1,\frac{T(x,y)}{\pi(y)}\}$, draw
      \[
      X_t\sim\frac{T(x,\cdot)-\min\{T(x,\cdot),\pi(\cdot)\}}{\delta(x)};
      \]
     \end{enumerate}
\end{enumerate}
Then similar to the proof of Lemma \ref{le:le1}, the above procedure is valid and the two Markov chains $X_t$ and $X'_t$ have the same transition kernel $T$, but have initial distribution $\delta_x$ and $\pi$, respectively.
Moreover,
\begin{align*}
P(X_t\neq X'_t|X_{1},X'_{1},\ldots,X_{t-1},X'_{t-1})\leq& \sup_x \bigg\{1-\int\min\{1,\frac{T(x,y)}{\pi(y)}\}\pi(y)dy\bigg\}\\
=& \sup_x\delta(x)\leq \rho.
\end{align*}
Therefore, we have
\begin{align*}
   ||T^t(x,\cdot)-\pi||_1\leq P(X_1\neq X'_1,\ldots,X_t\neq X'_t)\leq \rho^t.
\end{align*}
\end{proof}

Note that condition \eqref{eq:ul1} in the above lemma is weaker than the minorization condition \citep{Meyn1993} for proving uniform ergodicity with rate function $r(t)=\rho^t$. Minorization condition assumes that there exists a probability measure $\nu$ such that,
\begin{align}
    T(x,y)\geq (1-\rho) \nu(y), \forall x,y\in\mathcal{X}.\label{eq:mino}
\end{align}
In practice, there is no rule on how to choose such measure $\nu$. To see that \eqref{eq:ul1} is weaker, first note that if \eqref{eq:mino} holds, then by the stationarity of $\pi$,
\begin{align*}
    \pi(y)=\int T(x,y)\pi(x)dx\geq (1-\rho)\nu(y)\int \pi(x)dx=(1-\rho)\nu(y).
\end{align*}
Therefore, for any $x\in\mathcal{X}$, we have
\begin{align*}
    ||T(x,\cdot)-\pi||_1\leq & ||T(x,\cdot)-(1-\rho)\nu||_1+||\pi-(1-\rho)\nu||_1\\
    =& \int \big[T(x,y)-(1-\rho)\nu(y)\big]dy+\int \big[\pi(y)-(1-\rho)\nu(y)\big]dy\\
    =& 1-(1-\rho)+1-(1-\rho)=2\rho.
\end{align*}
Therefore, condition \eqref{eq:ul1} can lead to a tighter MCMC convergence bound than the minorization condition. Using $\sup_x||T(x,\cdot)-\pi||_1$ in \eqref{eq:ul1} also provides a tighter bound than using the Dobrushin coefficient $\beta(T)=\sup_{x,y}||T(x,\cdot)-T(y,\cdot)||_1$, which is another tool used in studying the Markov chain convergence rate via operator theory. In fact, for any set $A\subset\mathcal{X}$
\begin{align*}
    \sup_x|T(x,A)-\pi(A)|=&\sup_x\bigg|\int_A\bigg\{\int_{\mathcal{X}}\pi(y)
    \big[T(x,z)-T(y,z)\big]dy\bigg\}dz\bigg|\\
    \leq &\int_A\bigg\{\int_{\mathcal{X}}\pi(y)\big|T(x,z)-T(y,z)\big|dy\bigg\}dz\\
    \leq &\ \beta(T) \pi(A),
\end{align*}
which implies that $\sup_x||T(x,\cdot)-\pi||_1\leq \beta(T)$.
Moreover, comparing to the minorization condition and Dobrushin coefficient, \eqref{eq:ul1} has a more intuitive explanation that the closer the transition kernel $T(x,\cdot)$ is to the stationary distribution, the faster the convergence of the Markov chain. Ideally, if $T(x,\cdot)=\pi(\cdot)$ for all $x\in\mathcal{X}$, then the Markov chain converges in one step. The converse of Lemma \ref{le:le2} is also true as shown in the following lemma, which implies that condition \eqref{eq:ul1} is also necessary for uniform ergodicity.

\begin{lemma}
If $T$ is uniformly ergodic, then there exists $\rho\in(0,1)$, such that
\begin{align}
    \sup_x||T(x,\cdot)-\pi||_1\leq 2\rho.\label{eq:ob}
\end{align}
\end{lemma}

\begin{proof}
By Lemma \ref{le:le1}, $T$ is uniformly ergodic. Therefore by Theorem 1.3 in \cite{Mengersen1996}, \eqref{eq:mino} holds for some $\rho\in(0,1)$ and probability measure $\nu$. Then by the arguments after Lemma \ref{le:le2},  \eqref{eq:ob} holds with the same $\rho$.
\end{proof}

When the condition \eqref{eq:ul1} does not hold, we can still get a bound by applying the above coupling techniques. More specifically, assume $\{X_t:t\geq0\}$ is a Markov chain with state space $\mathcal{X}$, transition kernel $T$ and initial distribution $p_0$ over $\mathcal{X}$. Recall that $\pi$ is the stationary measure associated with $T$. We define an accompanied transition kernel $T'$ as
\begin{align*}
    T'(x,y)=\frac{T(x,y)-\min\{T(x,y),\pi(y)\}}{\delta(x)},
\end{align*}
where $\delta(x)=\frac{1}{2}||T(x,\cdot)-\pi||_1$. Let $\{X'_t:t\geq0\}$ be another Markov chain with state space $\mathcal{X}$, transition kernel $T'$ and the same initial distribution $p_0$. The following lemma characterizes the convergence of $X_t$ via $\tilde{X}_t$.

\begin{lemma}\label{le:none}
With the above notations and definitions, we have the following result:
\begin{align*}
    ||T^t\circ p_0-\pi||_1\leq E\big\{\prod_{s=1}^t\delta(X'_s)\big\}.
\end{align*}
\end{lemma}

\begin{proof}
We construct a new Markov chain $\{\tilde{X}_t:t\geq0\}$ as follows:
\begin{enumerate}
  \item The state space of $\tilde{X}_t$ is $\tilde{\mathcal{X}}=\mathcal{X}\cup\{c\}$, where $c$ is an extended ``coffin'' state.
  \item For $t>0$: if $\tilde{X}_{t-1}\neq c$, then with probability $\delta(x)$, $\tilde{X}_t=X'_t$ and with probability $1-\delta(x)$, $\tilde{X}_t=c$; if $\tilde{X}_{t-1}=c$, then $\tilde{X}_t=c$. Therefore, $c$ is an absorbing state.
  \item $\tilde{X}_0$ is distributed according to $p_0$.
\end{enumerate}
Then by identifying the coupling $(X_t=X'_t)$ in the proof of Lemma \ref{le:le2} as going to the absorbing state $c$, we have
\[
 ||T^t\circ p_0-\pi||_1\leq P(\tilde{X}_t\neq c)=E\big\{\prod_{s=1}^t\delta(X'_s)\big\},
\]
since before being coupled, $X'_t$ in the proof of Lemma \ref{le:le2} is a Markov chain with transition kernel $T'$.
\end{proof}

The Markov chain $\tilde{X}_t$ in the above proof is known as the trapping model in physics, where before getting trapped, a particle moves according to the transition kernel $T'$ on $\mathcal{X}$ and every time the particle moves to a new location $y$, with probability $1-\delta(y)$, it will be trapped there forever. Generally, the upper bound in Lemma \ref{le:none} is not easy to compute. However, under the drift condition and an analogue of local minorization assumption \citep{Rosenthal1995}, we can obtain an explicit quantitative bound for MCMC convergence as indicated by the following theorem. The proof is omitted here, which is a combination of the result in Lemma \ref{le:none} and the proof of Theorem 5 in \cite{Rosenthal1995}.

\begin{theorem}\label{prop:p1}
Suppose a Markov chain has transition kernel $T$ and initial distribution $p_0$. Assume the following two conditions:
\begin{enumerate}
  \item (Analogue of local minorization condition) There exists a subset $C\subset\mathcal{X}$, such that for some $\rho<1$,
      \begin{align*}
        \sup_{x\in C}||T(x,\cdot)-\pi||_1\leq 2\rho.
      \end{align*}
  \item (Drift condition) There exist a function $V:\mathcal{X}\rightarrow[1,\infty)$ and constant $b$ and $\lambda\in(0,1)$, such that for all $x\in \mathcal{X}$,
      \begin{align*}
        \int T(x,z)V(z)dz\leq \lambda V(x)+b1_{C}(x).
      \end{align*}
\end{enumerate}
Then for any $j$, $1\leq j\leq t$,
\begin{align*}
    ||T^t\circ p_0-\pi||_1\leq \rho^j+\lambda^t B^{j-1}\bar{V},
\end{align*}
where $B=1+{b}/{\lambda}$ and $\bar{V}=\int V(z)p_0(z)dz$.
\end{theorem}
By optimizing the $j$ in the above theorem, we can obtain the following geometrically decaying bound on $||T^t\circ p_0-\pi||_1$, which is similar to \cite{Rosenthal1995}:
\begin{align*}
    ||T^t\circ p_0-\pi||_1\leq \bar{V}\tilde{\rho}^t,
    \text{with }\log\tilde{\rho}=\frac{\log\rho\log\lambda}{\log\rho-\log B}.
\end{align*}
This implies that the Markov chain with transition kernel $T$ is geometrically ergodic. Recall that a chain is geometrically ergodic if there is $\rho<1$, and constants $C_x$ for each $x\in\mathcal{X}$, such that for $\pi-$a.e. $x\in\mathcal{X}$,
\[
||T^t(x,\cdot)-\pi(\cdot)||_1\leq C_x\rho^t.
\]

\subsection{Convergence of sequential MCMC}\label{se:not}
SMCMC generates $L$ time-inhomogeneous Markov chains. To investigate its asymptotic properties, we need a notion of convergence. Existing literature on the convergence of MCMC or adaptive MCMC focuses on the case when the stationary distribution does not change with time. A nonadaptive MCMC algorithm is said to be converging if
\begin{align}\label{eq:L1a}
    ||Q^t\circ p_0-\pi||_1\rightarrow 0,\ \ \text{as } t\rightarrow\infty,
\end{align}
 where $||\cdot||$ is the $L_1$ norm, $Q$ is the time homogeneous transition kernel, $p_0$ is the initial distribution and $\pi$ is the unique stationary measure. However, for sequential MCMC, both the stationary distribution $\pi_t$ and the transition kernel $Q_t$ is changing over time. As an extension of \eqref{eq:L1a}, a stationary-distribution-varying Markov chain is said to be convergent if
 \begin{align}\label{eq:L1b}
    ||Q_t\circ\cdots\circ Q_1\circ p_0-\pi_t||_1\rightarrow 0,\ \ \text{as } t\rightarrow\infty.
\end{align}
In our case, $Q_t=T_t^{m_t}\circ J_t$, where $T_t$, $J_t$ and $m_t$ are defined in section \ref{se:al}.

\subsubsection{Constant parameter dimension $d_t$}\label{se:cmt}
We first focus on the case when the parameter size is fixed, i.e. $J_t$ is the identity map. The following theorem provides guarantees for the convergence of SMCMC under certain conditions. We will use the convention that
$$
\sum_{\emptyset}=0\text{ and }\prod_{\emptyset}=1.
$$

\begin{theorem}\label{thm:main}
Use the notations in section \ref{se:not}. Assume the following conditions:
\begin{enumerate}
  \item (Universal ergodicity) There exists $\epsilon_t\in(0,1)$, such that for all $t>0$ and $x\in\mathcal{X}$,
      \begin{align*}
        ||T_t(x,\cdot)-\pi_t||\leq 2\rho_t.
      \end{align*}
  \item (Stationary convergence) The stationary distribution $\pi_t$ of $T_t$ satisfies
      \[
      \alpha_t=\frac{1}{2}||\pi_{t}-\pi_{t-1}||_1\to 0.
      \]
\end{enumerate}
Let $\epsilon_t=\rho_t^{m_t}$.
Then for any initial distribution $\pi_0$,
\begin{align*}
    ||Q_t\circ\cdots\circ Q_1\circ \pi_0-\pi_t||_1\leq& \sum_{s=1}^t\bigg\{\prod_{u=s+1}^t\epsilon_u(1-\alpha_u)\bigg\}\epsilon_s\alpha_s.
\end{align*}
\end{theorem}

\begin{proof}
We will construct two time inhomogeneous Markov chains $\{X_{t,s}:s=1,\ldots,m_t,\ t\geq0\}$ and $\{X'_{t,s}:s=1,\ldots,m_t,\ t\geq0\}$, where a double index is used as the step indicator under the following order $(0,1)\to\cdots\to(0,m_0)\to(1,1)\to\cdots\to(1,m_1)\to(2,1)\to\cdots\to(2,m_2)\to\cdots$.
Let $\delta_t(x)=\frac{1}{2}||T_t(x,\cdot)-\pi_t||_1$. The two chains are constructed as follows: (note that $m_0=1$)
\begin{enumerate}
  \item $X_{0,1}\sim \pi_0$, $X'_{0,1}\sim\pi_0$;
  \item For $t\geq1$,
  \begin{enumerate}
    \item $s=1$. Let $X_{t-1,m_{t-1}}=x$ and $X'_{t-1,m_{t-1}}=x'$. Set $X_{t,1}=x$. With probability $\min\{1,\frac{\pi_{t}(x)}{\pi_{t-1}(x)}\}$, set $X'_{t,1}=x$; with probability $1-\min\{1,\frac{\pi_{t}(x)}{\pi_{t-1}(x)}\}$, draw
          $$
          X'_{t,1}\sim\frac{\pi_t(\cdot)-\min\{\pi_t(\cdot),\pi_{t-1}(\cdot)\}}
          {\alpha_t};
          $$
    \item $1<s\leq m_t$. Let $X_{t,s-1}=x$ and $X'_{t,s-1}=x'$.
     \begin{enumerate}
       \item if $x=x'$, choose $X_{t,s}=X'_{t,s}\sim T_t(x,\cdot)$;
       \item else, first choose $X'_{t,s}=y\sim T_t(x',\cdot)$,
       then with probability $\min\{1,\frac{T_t(x,y)}{\pi_t(y)}\}$, set $X_{t,s}=y$,
       with probability $1-\min\{1,\frac{T_t(x,y)}{\pi_t(y)}\}$, draw
      \[
      X_{t,s}\sim\frac{T_t(x,\cdot)-\min\{T_t(x,\cdot),\pi_t(\cdot)\}}{\delta_t(x)};
      \]
     \end{enumerate}
  \end{enumerate}
\end{enumerate}
The above construction combines those in Lemma \ref{le:le1} and \ref{le:le2}. By the argument in the proof of Lemma \ref{le:le1}, the construction for $s=1$ is valid. Moreover, if $(X_{t-1,m_{t-1}}=X'_{t-1,m_{t-1}})$, then the probability of $(X_{t,1}\neq X'_{t,1})$ is $\alpha_t$. Similarly, by the argument in the proof of Lemma \ref{le:le2}, the construction for $s>1$ is valid. Moreover, conditioning on $X_{t,s-1}$ and $X'_{t,s-1}$, the conditional probability of $(X_{t,1}\neq X'_{t,1})$ does not exceed $\rho_t$.

It can been seen that the marginal distribution of $X_{t,s}$ is $T_t^s\circ Q_{t-1}\circ\cdots\circ Q_1\circ \pi_0$, while the marginal distribution of $X'_{t,s}$ is $\pi_t$, for $s=1,\ldots,m_t$. Therefore,
\begin{align*}
    ||Q_{t}\circ\cdots\circ Q_1\circ \pi_0-\pi_t||_1\leq P(X_{t,m_t}\neq X'_{t,m_t}).
\end{align*}
Furthermore, we have
\begin{align*}
    P(X_{t,m_t}\neq X'_{t,m_t})=&\ P(X_{t-1,m_{t-1}}\neq X'_{t-1,m_{t-1}},X_{t,m_t}\neq X'_{t,m_t})\\
    &+P(X_{t-1,m_{t-1}}= X'_{t-1,m_{t-1}},X_{t,m_t}\neq X'_{t,m_t})\\
    \leq&\  P(X_{t-1,m_{t-1}}\neq X'_{t-1,m_{t-1}})\rho_t^{m_t}\\
    &+\big[1-P(X_{t-1,m_{t-1}}\neq X'_{t-1,m_{t-1}})\big]{\alpha_t}\rho_t^{m_t}\\
    =&\ \alpha_t\epsilon_t+\epsilon_t(1-{\alpha_t})
    P(X_{t-1,m_{t-1}}\neq X'_{t-1,m_{t-1}})\\
    \leq&\ \cdots\ \leq\sum_{s=1}^t\bigg\{\prod_{u=s+1}^t\epsilon_u(1-\alpha_u)\bigg\}\epsilon_s\alpha_s.
\end{align*}

Combining the above two inequalities, the theorem can be proved.
\end{proof}

The above theorem has a short proof if the conclusion is weaken to
\begin{align*}
    ||Q_t\circ\cdots\circ Q_1\circ \pi_0-\pi_t||_1
     \leq &\ 2\sum_{s=1}^{t} \bigg\{\prod_{u=s}^t\epsilon_t\bigg\}\alpha_s\rightarrow0, \text{ as } n\rightarrow\infty.
\end{align*}
In fact, by the universally ergodicity condition and Lemma \ref{le:le2}, for all $t>0$ and any probability distribution $p$,
\begin{align}\label{eq:dimi}
    ||Q_t\circ p-\pi_t||_1=||T_t^{m_t}\circ p-\pi_t||_1\leq \epsilon_t||p-\pi_t||_1.
\end{align}
Therefore, by recursively applying \eqref{eq:dimi}, we have
\begin{align*}
    ||Q_t\circ\cdots\circ Q_1\circ \pi_0-\pi_t||_1\leq&\epsilon_t
    ||Q_{t-1}\circ\cdots\circ Q_1\circ \pi_0-\pi_t||_1\\
    \leq &\epsilon_t
    ||Q_{t-1}\circ\cdots\circ Q_1\circ \pi_0-\pi_{t-1}||_1+\epsilon_t||\pi_t-\pi_{t-1}||_1\\
    \leq &\cdots\leq \sum_{s=1}^{t} \bigg\{\prod_{u=s}^t\epsilon_t\bigg\}||\pi_{s}-\pi_{s-1}||_1.
\end{align*}

If $m_t$ in the algorithm is chosen large enough so that
\begin{align}\label{eq:Tm}
    \sup_{x}||T_t^{m_t}(x,\cdot)-\pi_t||_1\leq 2(1-\epsilon),
\end{align}
then
\begin{align*}
    ||Q_t\circ\cdots\circ Q_1\circ \pi_0-\pi_t||_1\leq&\sum_{s=1}^{t} \bigg\{\prod_{u=s}^t\epsilon_t\bigg\}\alpha_s\\
     \leq &\sum_{s=1}^{t} (1-\epsilon)^{t+1-s}\alpha_s\rightarrow0, \text{ as } n\rightarrow\infty.
\end{align*}

In practice, we can choose $m_t$ as in section \ref{se:mt}, which provides good approximations to \eqref{eq:Tm}.
Although $T_t$ are required to be universally ergodic in the theorem, it might be possible to weaken the conditions to those in Theorem \ref{prop:p1} with direct application of the coupling techniques in the proofs of Lemma \ref{le:le1} and Lemma \ref{le:le2}. The current paper only focuses on the universally ergodic case for conciseness and easy exhibition. Condition 2 is intuitively reasonable and can be verified for many problems. In this subsection, we provide such a verification for regular parametric cases in Lemma \ref{le:le3} below, where the Bernstein von-Mises theorem holds. In the next subsection when $d_t$ is allowed to grow in $t$, we provide a verification for general models that may not have $n^{-1/2}$ convergence rate or Gaussian limiting distributions; for example, nonparametric models.

\begin{lemma}\label{le:le3}
Suppose that we have a parametric model with regular likelihood function such that the following Berstein Von-Mises theorem holds:
\begin{align*}
    \big|\big|\pi_n-N\big(\hat{\theta}_n,\frac{1}{n}I^{-1}\big)\big|\big|_1\rightarrow 0, \text{ in probability},
\end{align*}
where $N(\mu,\Sigma)$ is the multivariate normal distribution with mean $\mu$ and covariance $\Sigma$, $\hat{\theta}_n$ is the maximum likelihood estimator, $\pi_n$ is the posterior distribution with $n$ observations $Y_1,\ldots,Y_n$ and $I$ is the Fisher information matrix. Moreover assume that $||\hat{\theta}_n-\hat{\theta}_{n-1}||=O_p(n^{-1})$, where $||\cdot||$ is the Euclidean norm.
If one observation is added at each time, so that time $t$ is the sample size $n$, then the stationary convergence condition in Theorem \ref{thm:main} holds.
\end{lemma}

\begin{proof}
Note that the $L_1$ distance $||p-q||_1$ between any two densities $p$ and $q$ is bounded by $H(p,q)/\sqrt{2}$, where $H^2(p,q)=\int|\sqrt{p}-\sqrt{q}|^2$ is the square of the Hellinger distance. Moreover, for two normal distributions, $N(\mu_1,\sigma_1^2)$ and $N(\mu_2,\sigma_2^2)$, we have
\begin{align*}
    H^2\big(N(\mu_1,\sigma_1^2),N(\mu_2,\sigma_2^2)\big)=& 1-\sqrt{\frac{2\sigma_1\sigma_2}{\sigma_1^2+\sigma_2^2}}e^{-\frac{1}{4}
    \frac{(\mu_1-\mu_2)^2}{\sigma_1^2+\sigma_2^2}}.
\end{align*}
Therefore we have,
\begin{align*}
    ||\pi_n-\pi_{n-1}||_1\leq & \big|\big|\pi_n-N\big(\hat{\theta}_n,\frac{1}{n}I^{-1}\big)\big|\big|_1+
    \big|\big|\pi_{n-1}-N\big(\hat{\theta}_{n-1},\frac{1}{n-1}I^{-1}\big)\big|\big|_1\\
    &+\big|\big|N\big(\hat{\theta}_n,\frac{1}{n}I^{-1}\big)-
    N\big(\hat{\theta}_{n-1},\frac{1}{n-1}I^{-1}\big)\big|\big|_1\\
    =&o_p(1)+1-\frac{n^{1/4}(n-1)^{1/4}}{(n-1/2)^{1/2}}e^{-\frac{n(n-1)}{8n-4}
    O_p(1/n^2)}\\
    =&o_p(1)+o_p(1)=o_p(1).
\end{align*}
\end{proof}

Note that the condition $||\hat{\theta}_n-\hat{\theta}_{n-1}||=O_p(n^{-1})$ holds for regular parametric models. For simplicity, we illustrate this for a one dimensional case. Let $l(y,\theta)=\log P(y|\theta)$ be the log likelihood function, $\dot{l}(y,\theta)$ be its derivative with respect to $\theta$ and $\ddot{l}(y,\theta)$ its second order derivative. Applying a Taylor expansion of $\sum_{i=1}^n\dot{l}(Y_i,\theta)$ around $\hat{\theta}_{n-1}$, we obtain for $\theta$ in a small neighborhood around $\hat{\theta}_{n-1}$,
\begin{align*}
    \sum_{i=1}^n\dot{l}(Y_i,\theta)-\sum_{i=1}^n\dot{l}(Y_i,\hat{\theta}_{n-1})
    =\sum_{i=1}^n\ddot{l}(Y_i,\hat{\theta}_{n-1})(\theta-\hat{\theta}_{n-1})
    +O\big(n(\theta-\hat{\theta}_{n-1})^2\big)
\end{align*}
Plugging in $\theta$ with $\hat{\theta}_{n}$ and using $\sum_{i=1}^n\dot{l}(Y_i,\hat{\theta}_{n})=0$, $\sum_{i=1}^{n-1}\dot{l}(Y_i,\hat{\theta}_{n-1})=0$ and  $\sum_{i=1}^n\ddot{l}(Y_i,\hat{\theta}_{n-1})\rightarrow nI$ in probability, we obtain
\begin{align*}
    -\dot{l}(Y_n,\hat{\theta}_{n-1})
    =nI(\hat{\theta}_{n}-\hat{\theta}_{n-1})
    +o_p\big(n|\hat{\theta}_{n}-\hat{\theta}_{n-1}|\big).
\end{align*}
Finally we reach
\begin{align*}
 |\hat{\theta}_{n}-\hat{\theta}_{n-1}|=
 -[1+o_p(1)]\ (nI)^{-1}\ \dot{l}(Y_n,\hat{\theta}_{n-1})=O_p(n^{-1}).
\end{align*}

\subsubsection{Increasing parameter dimension $d_t$}\label{se:ipd}
Recall that the parameter at $t$ can be written as $\theta^{(t)}=(\theta^{(t-1)},\eta_{t})$. Consider the $J_t$ satisfying \eqref{eq:cJt} in section \ref{se:Jt} and $Q_t=J_t\circ T_t^{m_t}$.  The following lemma links the approximation errors before and after applying the jumping kernel $J_t$.
\begin{lemma}\label{le:Jt}
For any probability density $p(\cdot)$ for $\theta^{(t-1)}$, the following holds:
\begin{align*}
    ||\pi_{t}-J_t\circ p||_1\leq &||\pi_{t-1}-p||_1
    +\sup_{\theta^{(t-1)}\in\bbR^{d_{t-1}}}||\pi_{t}(\cdot|\theta^{(t-1)})-J_t(
    \theta^{(t-1)},\cdot)||_1,
\end{align*}
where the $\pi_t$ in the second term of the right hand side stands for the marginal posterior of $\theta^{(t-1)}$ at time $t$.
\end{lemma}

\begin{proof}
By factorization of joint probability, we have
\begin{align*}
    ||\pi_{t}-J_t\circ\hat{\pi}_{t-1}||_1=&\int |\pi_{t}(\theta^{(t-1)})\pi_t(\eta_t|\theta^{(t-1)})-p(\theta^{(t-1)})
    J_t(\theta^{(t-1)},\eta_t)|d\theta^{(t-1)}d\eta_t\\
    \leq &\int \pi_{t}(\theta^{(t-1)})|\pi_t(\eta_t|\theta^{(t-1)})-
    J_t(\theta^{(t-1)},\eta_t)|d\eta_t d\theta^{(t-1)}\\
    &+\int |\pi_{t}(\theta^{(t-1)})-p(\theta^{(t-1)})|
    J_t(\theta^{(t-1)},\eta_t)d\theta^{(t-1)}d\eta_t\\
    \leq & \sup_{\theta^{(t-1)}\in\bbR^{d_{t-1}}}||\pi_{t}(\cdot|\theta^{(t-1)})-J_t(
    \theta^{(t-1)},\cdot)||_1+||\pi_{t}-p||_1.
\end{align*}
\end{proof}

If a Gibbs or slice sampling step is applied as $J_t$, then the last term in the above lemma vanishes. With Lemma \ref{le:Jt}, we can prove the following analogue of Theorem \ref{thm:main} for the increasing $d_t$ scenario.

\begin{theorem}\label{thm:mainb}
Assuming the following conditions:
\begin{enumerate}
  \item (Universal ergodicity) There exists $\epsilon\in(0,1)$, such that for all $t>0$ and $x\in\mathcal{X}$,
      \begin{align*}
        ||T_t(x,\cdot)-\pi_t||_1\leq 2\rho_t.
      \end{align*}
  \item (Stationary convergence) The stationary distribution $\pi_t$ of $T_t$ satisfies
      \[
      \alpha_t=\frac{1}{2}||\pi_{t}-\pi_{t-1}||_1\to 0,
       \]
       where $\pi_t$ is the marginal posterior of $\theta^{(t-1)}$ at time $t$ in $\alpha_t$.
  \item (Jumping consistency) For a sequence of $\lambda_t\to 0$, the following holds:
      \begin{align*}
      \sup_{\theta^{(t-1)}\in\bbR^{d_{t-1}}}||\pi_{t}(\cdot|\theta^{(t-1)})-J_t(
        \theta^{(t-1)},\cdot)||_1\leq 2\lambda_t.
      \end{align*}
\end{enumerate}
Let $\epsilon_t=\rho_t^{m_t}$. Then for any initial distribution $\pi_0$,
\begin{align*}
   ||Q_t\circ\cdots\circ Q_1\circ \pi_0-\pi_t||_1\leq \sum_{s=1}^t\bigg\{\prod_{u=s}^t\epsilon_u\bigg\}
   (\alpha_s+\lambda_s).
\end{align*}
\end{theorem}

\begin{proof}
The proof is almost the same as that of Theorem \ref{thm:main}. The only difference occurs in the constructions of $X_{t,s}$ and $X'_{t,s}$ for $s=1$, which is provided in the following.

When $t\geq1$ and $s=1$, let $X_{t-1,m_{t-1}}=x$ and $X'_{t-1,m_{t-1}}=x'$. Draw $X_{t,1}\sim J_t(x,\cdot)$. With probability $\min\{1,\frac{\pi_{t}(x)}{J_t\circ \pi_{t-1}(x)}\}$, set $X'_{t,1}=x$; with probability $1-\min\{1,\frac{\pi_{t}(x)}{J_t\circ \pi_{t-1}(x)}\}$, draw
$$
  X'_{t,1}\sim\frac{\pi_t(\cdot)-\min\{\pi_t(\cdot),J_t\circ \pi_{t-1}(x)(\cdot)\}}
  {\tilde{\alpha}_t},
$$
where $\tilde{\alpha}_t=\frac{1}{2}||\pi_{t}-J_t\circ \pi_{t-1}||_1$ is the probability of $(X_{t,1}\neq X'_{t,1})$ conditioning on $(X_{t-1,m_{t-1}}=X'_{t-1,m_{t-1}})$. Moreover, by Lemma \ref{le:Jt}, we have $\tilde{\alpha}_t\leq \alpha_t+\lambda_t$.
\end{proof}

Similarly, if we choose $m_t$ such that
\begin{align*}
    \sup_{x}||T_t^{m_t}(x,\cdot)-\pi_t||_1\leq 2(1-\epsilon),
\end{align*}
then
\[
||Q_t\circ\cdots\circ Q_1\circ \pi_0-\pi_t||_1\rightarrow0, \text{ as } n\rightarrow\infty.
\]

An increasing parameter dimension often occurs in Bayesian nonparametric models, such as Dirichlet mixture models and Gaussian process regressions. The following lemma is a counterpart of Lemma \ref{le:le3} for general models that may not have $n^{-1/2}$ convergence rate or normal as limiting distribution for the parameters.
A function $f$ defined on a Banach space $(V,||\cdot||)$ is said to be Fr\'echet differentiable at $v\in V$ if there exists a bounded linear operator $A_v:V\rightarrow\bbR$ such that
\begin{align*}
    f(v+h)=f(v)+A_v(h)+o(||h||), \text{ as }||h||\to0,
\end{align*}
where $A_v$ is called the Fr\'echet derivative of $f$. For $V$ being a Euclidean space, Fr\'echet differentiability is equivalent to the usual differentiability.
The proof utilizes the notion of posterior convergence rate \citep{Ghosal2000} and Fr\'echet differentiability.

\begin{lemma}\label{le:le4}
Consider a Bayesian model $\mathcal{P}=\{P_{\theta}:\theta\in\Theta\}$ with a prior measure $\Pi$ on a Banach space $(\Theta,||\cdot||)$, where the parameter space $\Theta$ can be infinite dimensional. Let $p_{\theta}$ be the density of $P_{\theta}$.
Assume the following conditions:
\begin{enumerate}
  \item the posterior convergence rate of the Bayesian model is at least $\epsilon_n\rightarrow0$ as $n\to\infty$, i.e. the posterior satisfies
\begin{align*}
    \Pi(||\theta-\theta_0||>M\epsilon_n|Y_1,\ldots,Y_n)\rightarrow 0,\text{ in probability,}
\end{align*}
where $Y_1,\ldots,Y_n$ is the observation sequence generated according to  $P_{\theta_0}$, $M>0$ is a constant.
  \item Assume that
  \begin{align*}
    &\max\big[E_{[\theta|Y_1,\ldots,Y_n]}\{p_{\theta}(Y)
    I(||\theta-\theta_0||>M\epsilon_n)\},\\
    &E_{[\theta|Y_1,\ldots,Y_n]}\{\log p_{\theta}(Y)
    I(||\theta-\theta_0||>M\epsilon_n)\}\big]\to 0 \text{ in probability},
  \end{align*}
  where $Y\sim P_{\theta_0}$ is independent of $Y_1,\ldots,Y_n$ and the expectation is taken with respect to the posterior distribution $\Pi(\theta|Y_1,\ldots,Y_n)$ for $\theta$.
  \item Also assume that the log likelihood function $\log p_{\theta}(y)$ is Fr\'echet differentiable at $\theta_0$ with a Fr\'echet derivative $A_{v,y}$ satisfying $E_{\theta_0}||A_{v,Y}||<\infty$, where $||\cdot||$ is the induced operator norm and the expectation is taken with respect to $Y\sim P_{\theta_0}$.
\end{enumerate}
 Then
\begin{align*}
    ||\pi(\cdot|Y_1,\ldots,Y_n)-\pi(\cdot|Y_1,\ldots,Y_{n-1})||_1\to 0, \text{ as }n\to\infty.
\end{align*}
\end{lemma}

\begin{proof}
Recall that the Kullback-Leibler (KL) divergence is defined as
\begin{align*}
    K(p,q)=\int p(\theta)\log\frac{p(\theta)}{q(\theta)}d\theta,
\end{align*}
where $f$ and $g$ are two pdfs on $\Theta$. We will use the following relationship between KL divergence and $L_1$ norm:
\begin{align}\label{eq:l1K}
    ||p-q||_1\leq 2\sqrt{K(p,q)}.
\end{align}
Use the shorthand $\pi_n$ for the posterior density $\pi(\cdot|Y_1,\ldots,Y_n)$ for $\theta$. By definition,
\begin{align*}
    \pi_n(\theta)=\frac{\exp\{\sum_{i=1}^n  l_i(\theta)\}\pi(\theta)}{\int_{\Theta}\exp\{\sum_{i=1}^n l_i(\theta)\}\pi(\theta)d\theta},
\end{align*}
where $l_i(\theta)=\log p_{\theta}(Y_i)$ is the log likelihood for the $i$th observation and $\pi$ is the prior for $\theta$. Moreover,
\begin{align*}
    \log\frac{\pi_{n-1}(\theta)}{\pi_n(\theta)}=&-l_n(\theta)+
    \log\bigg\{\int_{\Theta}\frac{\exp\{\sum_{i=1}^{n-1} l_i(\theta)\}\pi(\theta)}{\int_{\Theta}\exp\{\sum_{i=1}^{n-1} l_i(\theta)\}\pi(\theta)d\theta}\exp\{l_n(\theta)\}d\theta\bigg\}\\
    = &-l_n(\theta)+\log E_{[\theta|Y_1,\ldots,Y_{n-1}]} \exp\{l_n(\theta)\},
\end{align*}
where $E_{[\theta|Y_1,\ldots,Y_{n-1}]}$ is the expectation with respect to the posterior distribution $\Pi(\theta|Y_1,\ldots,Y_{n-1})$.
Therefore, we obtain:
\begin{align}\label{eq:KL}
    K(\pi_{n-1},\pi_{n})=&\int_{\Theta}\pi_{n-1}(\theta)
    \log\frac{\pi_{n-1}(\theta)}{\pi_n(\theta)}d\theta\nonumber\\
    =& \log E_{[\theta|Y_1,\ldots,Y_{n-1}]} \exp\{l_n(\theta)\}-E_{[\theta|Y_1,\ldots,Y_{n-1}]} \{l_n(\theta)\}.
\end{align}
By the third condition, we have that for any $||\theta-\theta_0||\leq M\epsilon_n$,
\begin{align*}
    l_n(\theta)=l_n(\theta_0)+O_{P_{\theta_0}}(\epsilon_n).
\end{align*}
Combining the above with the second condition, we have
\begin{align*}
    E_{[\theta|Y_1,\ldots,Y_{n-1}]} \exp\{l_n(\theta)\}=&
    E_{[\theta|Y_1,\ldots,Y_{n-1}]}\big\{ \exp\{l_n(\theta)\}I(||\theta-\theta_0||\leq M\epsilon_n)\big\}+o_{P_{\theta_0}}(1)\\
    =& \exp\{l_n(\theta_0)\}+o_{P_{\theta_0}}(1).
\end{align*}
Similarly, we have
\begin{align*}
    E_{[\theta|Y_1,\ldots,Y_{n-1}]}\{l_n(\theta)\}=
   l_n(\theta_0)+o_{P_{\theta_0}}(1).
\end{align*}
Combining the above two with \eqref{eq:l1K} and \eqref{eq:KL}, we obtain
\begin{align*}
    ||\pi(\cdot|Y_1,\ldots,Y_n)-\pi(\cdot|Y_1,\ldots,Y_{n-1})||_1\to 0, \text{ as }n\to\infty.
\end{align*}
\end{proof}

The second assumption strengthens the first assumption in terms of the tail behavior of the posterior distributions and can be implied by the first if both $p_{\theta}(y)$ and $\log p_{\theta}(y)$ are uniformly bounded; for example, when $\Theta$ is compact.
The following corollary is an easy consequence of the above lemma by using the inequality $|\int f(x)dx|\leq\int|f(x)|dx$.
\begin{corollary}\label{cor:mar}
Let $\xi$ be a $d_0$ dimensional component of $\theta$. Denote the marginal posterior of $\xi$ by $\pi_{\xi}(\cdot|Y_1,\ldots,Y_n)$. Then under the conditions in Lemma \ref{le:le4}, we have
\begin{align*}
    ||\pi_{\xi}(\cdot|Y_1,\ldots,Y_n)-\pi_{\xi}(\cdot|Y_1,\ldots,Y_{n-1})||_1
    \to 0, \text{ as }n\to\infty.
\end{align*}
\end{corollary}

In the case when $T_t$ corresponds to the transition kernel of a Gibbs sampler,
we can consider the marginal convergence of some fixed $d_0$ dimensional component
$\xi$ of $\theta$, for example, for $\theta$ in function spaces, $\xi$ can be the
 evaluations $\theta(x_1,\ldots,x_{d_0})$ on $d_0$ fixed points $x_1,\ldots,x_{d_0}$
  in the domain of $\theta$. Due to the special structure of the graphical
  representation for a Gibbs sampler, the process of $\{\xi_s:s\geq0\}$ obtained
  by marginalizing out other parameters in the Gibbs sampler with transition
  kernel $T_t$ is still a Markov chain with another transition kernel $T_{\xi,t}$
  defined on $\bbR^{d_0}\times\bbR^{d_0}$. Therefore, with this marginalized process
   for $\xi$, we can combine Theorem \ref{thm:main} and Corollary \ref{cor:mar}
   to prove the marginal convergence of the posterior for the fixed dimensional
   parameter $\xi$ under the new transition kernels $T_{\xi,t}$'s. To ensure the
    convergence of this marginal chain, $m_t$ can also be chosen by the procedures
     in section \ref{se:mt}, but only including the components of $\xi$ in the
     calculations of \eqref{eq:ftk}.

\subsection{Weakening the universal ergodicity condition}
Both Theorem \ref{thm:main} and \ref{thm:mainb} rely on the strong condition of universal ergodicity. In this subsection, we generalize these results to hold under the weaker geometrically ergodic condition. We will use the following sufficient condition
for geometric ergodicity \citep{roberts1997} for an irreducible, aperiodic Markov chain
with transition kernel $T$:
there exists a $\pi$-a.e.-finite measurable function $V:\mathcal{X}\to[1,\infty]$, which
may be taken to satisfy $\pi(V^k)<\infty$ for any $j\in\mathds{N}$, such that for some $\rho<1$,
\begin{align}\label{eq:dge}
    ||T^t(x,\cdot)-\pi(\cdot)||_V\leq V(x)\rho^t,\quad x\in\mathcal{X},\quad t\in\mathds{N},
\end{align}
where $||\mu(\cdot)||_V=\sup_{|f|\leq V}|\mu(f)|$ for any signed measure $\mu$. When $V\equiv1$, we return to the uniform ergodic case.
The following lemma generalizes Lemma \ref{le:le1} and \ref{le:le2} to geometrically ergodic chains.

\begin{lemma}\label{le:ge}
Let $\{X_t\}$ be a Markov chain on $\mathcal{X}$, with transition kernel $T$ and stationary
distribution $\pi$. If there exists a function $V:\mathcal{X}\mathds{}\to[1,\infty)$ and
$\rho\in(0,1)$ such that for all $x\in\mathcal{X}$,
      \begin{align}\label{eq:cge}
        ||T(x,\cdot)-\pi(\cdot)||_V\leq V(x)\rho,
      \end{align}
      then $\{X_t\}$ is geometrically ergodic. Moreover, for any initial distribution $p_0$,
      we have
      \begin{align*}
        ||T^t\circ p_0-\pi||_V\leq \rho^t||p_0-\pi||_V,\quad x\in\mathcal{X},\quad t\in\mathds{N}.
      \end{align*}
\end{lemma}

\begin{proof}
For a kernel $K(x,y)$ on $\mathcal{X}\times\mathcal{X}$, we define
\begin{align*}
    |||K|||_V=\sup_{x\in\bbR^d}\frac{||K(x,\cdot)||_V}{V(x)}
    =\sup_{x\in\bbR^d}\sup_{|f|\leq V}\frac{|(Kf)(x)|}{V(x)}.
\end{align*}
It is easy to verify that $|||\cdot|||_V$ satisfies the triangle inequality.
By viewing $\pi(x,y)=\pi(y)$ as a kernel on $\mathcal{X}\times\mathcal{X}$,
we have $|||T-\pi|||_V\leq\rho$. Moreover, for any $t\in\mathds{N}$, we have,
\begin{align*}
    |||T^t-\pi|||_V=&\sup_{x\in\mathcal{X}}\sup_{|f|\leq V}\frac
    {\big|\{(T-\pi)(T^{t-1}-\pi)f\}(x)\big|}{V(x)}\\
    =& |||T^{t-1}-\pi|||_V\sup_{x\in\mathcal{X}}\sup_{|f|\leq V}\frac
    {|\{(T-\pi)g_f\}(x)|}{V(x)},
\end{align*}
with $g_f(x)=\{(T^{t-1}-\pi)f\}(x)/|||T^{t-1}-\pi|||_V$. By the definition of $|||\cdot|||_V$,
we have $|g_f|\leq V$ for any $f$ satisfying $|f|\leq V$. Combining the above arguments, we obtain
\begin{align*}
    |||T^t-\pi|||_V\leq &|||T^{t-1}-\pi|||_V\cdot|||T-\pi|||_V\\
    \leq & \rho |||T^{t-1}-\pi|||_V\\
    \leq &\cdots\leq \rho^{t}.
\end{align*}
This implies geometric ergodicity, i.e.
\begin{align*}
    ||T^t(x,\cdot)-\pi(\cdot)||_V\leq V(x)\rho^t,\quad x\in\mathcal{X},\quad t\in\mathds{N}.
\end{align*}
For the second part, by the stationarity of $\pi$, we have
\begin{align*}
    ||T^t\circ p_0-\pi||_V=& \sup_{|f|\leq V}\int_{\mathds{X}}\{p_0(x)-\pi(x)\}
    \{(T^t-\pi)f\}(x)dx\\
    \leq & \int_{\mathds{X}}|p_0(x)-\pi(x)|\ V(x)\sup_{|f|\leq V}\frac{|\{(T^t-\pi)f\}(x)|}{V(x)}dx\\
    \leq & \rho^t||p_0-\pi||_V.
\end{align*}

\end{proof}

By taking $t=1$ in \eqref{eq:dge}, \eqref{eq:cge} is also a necessary condition for geometric ergodicity. Therefore, the above lemma provides a necessary and sufficient condition for geometric ergodicity, which extends Lemma \ref{le:le2}.
By the above lemma, we can generalize Theorem \ref{thm:main} as follows, where $d_t=d$, for any $t$.
\begin{theorem}\label{thm:gmain}
Assuming the following conditions:
\begin{enumerate}
  \item (Geometric ergodicity) There exists a function $V:\bbR^{d}\to[1,\infty)$, $C>0$
     and $\rho_t\in(0,1)$, such that $\pi_t(V^2)=E_{\pi_t}V^2\leq C$ for any $t$
       and for all $x\in\bbR^{d}$,
      \begin{align*}
        ||T_t(x,\cdot)-\pi_t(\cdot)||_V\leq V(x)\rho_t.
      \end{align*}
  \item (Stationary convergence) The stationary distribution $\pi_t$ of $T_t$ satisfies
      \begin{align*}
        \alpha_t={2\sqrt{C}}d_H(\pi_t,\pi_{t-1})\to0,
      \end{align*}
      where $d_H$ is the Hellinger distance defined by $d^2(\mu,\mu')=\int(\mu^{1/2}(x)
      -\mu'^{1/2}(x))^2dx$.
\end{enumerate}
Let $\epsilon_t=\rho_t^{m_t}$. Then for any initial distribution $\pi_0$,
\begin{align*}
   ||Q_t\circ\cdots\circ Q_1\circ \pi_0-\pi_t||_1\leq \sum_{s=1}^t\bigg\{\prod_{u=s}^t\epsilon_u\bigg\}
   \alpha_s.
\end{align*}
\end{theorem}

\begin{proof}
By Lemma \ref{le:ge}, for any distribution $p_0$ on $\bbR^d$ and any $t\in\mathds{N}$, we have
\begin{align*}
    ||T_t^{m_t}\circ p_0-\pi_t||_V\leq \rho_t^{m_t}||p_0-\pi_t||_V.
\end{align*}
Therefore, we have
\begin{align*}
    ||Q_t\circ\cdots\circ Q_1\circ\pi_0-\pi_t||_V\leq &\epsilon_t||Q_{t-1}\circ\cdots\circ Q_1\circ\pi_0-\pi_t||_V\\
    \leq &\epsilon_t||Q_{t-1}\circ\cdots\circ Q_1\circ\pi_0-\pi_{t-1}||_V+
    \epsilon_t||\pi_{t}-\pi_{t-1}||_V.
\end{align*}
By Cauchy's inequality,
\begin{align*}
  ||\pi_{t}-\pi_{t-1}||_V=&\int_{\bbR^d} |\pi_{t}(x)-\pi_{t-1}(x)|V(x)dx\\
  \leq & d_H(\pi_t,\pi_{t-1})\bigg[\int_{\bbR^d} \{\pi_{t}^{1/2}(x)+\pi_{t-1}^{1/2}(x)\}^2V^2(x)dx\bigg]^{1/2}\\
  \leq & 2\sqrt{C}d_H(\pi_t,\pi_{t-1})=\alpha_t.
\end{align*}
Combining the above two inequalities, we obtain
\begin{align*}
    ||Q_t\circ\cdots\circ Q_1\circ\pi_0-\pi_t||_V\leq &\epsilon_t||Q_{t-1}\circ\cdots\circ Q_1\circ\pi_0-\pi_{t-1}||_V+
    \alpha_t\epsilon_t\\
    \leq&\cdots\leq \sum_{s=1}^t\bigg\{\prod_{u=s}^t\epsilon_u\bigg\}
   \alpha_s.
\end{align*}
Finally, the theorem can be proved by noticing that $||\mu||_1=\sup_{||f||\leq1}|\mu(f)|\leq\sup_{||f||\leq V}|\mu(f)|=||\mu||_V$ for any signed measure $\mu$.
\end{proof}

The first condition in the theorem is a uniform geometric ergodic condition on the collection $\{T_t:t\in\mathds{N}\}$ of transition kernels, where a common potential $V$ exists. The second condition is true for those $\pi_t$'s in Lemma \ref{le:le3} and \ref{le:le4}. In fact, Lemma \ref{le:le3} uses the inequality $||\pi_t-\pi_{t-1}||_1\leq d_H(\pi_t,\pi_{t-1})$ and proves $d_H(\pi_t,\pi_{t-1})\to 0$. Lemma \ref{le:le4} proves $||\pi_t-\pi_{t-1}||_1\leq 2\sqrt{K(\pi_t,\pi_{t-1})}\to 0$, where $K(p,q)$ is the Kullback-Leibler divergence and satisfies $d_H(p,q)^2\leq K(p,q)$ for any probability densities $p$ and $q$.

Similarly, we have the following counterpart for Theorem \ref{thm:mainb} under geometrically ergodic condition.

\begin{theorem}\label{thm:gmainb}
Assuming the following conditions:
\begin{enumerate}
  \item (Geometric ergodicity) For each $t$, there is a function $V_t:\bbR^{d_t}\to[1,\infty)$, $C>0$
     and $\rho_t\in(0,1)$, such that:
      \begin{enumerate}
        \item $\pi_t(V_t^2)=E_{\pi_t}V_t^2\leq C$ for any $t$;
        \item $E_{\pi_t}[V_t(\theta^{(t)})|\theta^{(t-1)}]=V_{t-1}(\theta^{(t-1)})$,
            where $\theta^{(t)}=(\theta^{(t-1)},\eta_t)$;
        \item for all $x\in\bbR^{d_t}$,
      \begin{align*}
        ||T_t(x,\cdot)-\pi_t(\cdot)||_{V_t}\leq V_t(x)\rho_t.
      \end{align*}
      \end{enumerate}
  \item (Stationary convergence) The stationary distribution $\pi_t$ of $T_t$ satisfies
      \begin{align*}
        \alpha_t={2\sqrt{C}}d_H(\pi_t,\pi_{t-1})\to0,
      \end{align*}
      where $\pi_t$ is the marginal posterior of $\theta^{(t-1)}$ at time $t$ in $\alpha_t$.
  \item (Jumping consistency) For a sequence of $\lambda_t\to 0$, the following holds:
      \begin{align*}
      \sup_{\theta^{(t-1)}\in\bbR^{d_{t-1}}}||\pi_{t}(\cdot|\theta^{(t-1)})-J_t(
        \theta^{(t-1)},\cdot)||_{\tilde{V}_t}\leq \lambda_t,
      \end{align*}
      where $\tilde{V}_t$ is defined on $\bbR^{d_t-d_{t-1}}$ by $\tilde{V}_t(\eta_t)=\int_{\bbR^{d_{t-1}}}V_t(\theta^{(t-1)},\eta_t)d\theta^{(t-1)}$.
\end{enumerate}
Let $\epsilon_t=\rho_t^{m_t}$. Then for any initial distribution $\pi_0$,
\begin{align*}
   ||Q_t\circ\cdots\circ Q_1\circ \pi_0-\pi_t||_1\leq \sum_{s=1}^t\bigg\{\prod_{u=s}^t\epsilon_u\bigg\}
   (\alpha_s+\lambda_s).
\end{align*}
\end{theorem}

\subsection{Relationship between Markov chain convergence rate and the autocorrelation function}\label{se:semt}
The convergence results in the previous two subsections are primarily based on a coupling technique, which can provide explicitly quantitative convergence bounds for computation. The arguments in this subsection will mainly utilize functional analysis and operator theory, which can reveal the relationship between convergence rate and maximal correlation between two states in the Markov chain. For background details, refer to chapter 12 in \cite{Jun2001}.

For a time homogeneous Markov chain $\{X_t:t=0,1\ldots\}$ with transition kernel $T(x,y)$ and stationary distribution $\pi$,
consider the space of all mean zero and finite variance functions under $\pi$
\begin{align*}
    L_0^2(\pi)=\bigg\{h(x):\int h^2(x)\pi(x)dx<\infty\ \text{ and }\
    \int h(x)\pi(x)dx=0\bigg\}.
\end{align*}
Being equipped with the inner product
\begin{align}\label{eq:inner}
    \langle h, g\rangle=E_{\pi}\{h(x)\cdot g(x)\},
\end{align}
$L_0^2(\pi)$ becomes a Hilbert space. On $L_0^2(\pi)$, we can define two operators, called forward and backward operators, as
\begin{align*}
    Fh(x)\triangleq &\int h(y)T(x,y)dy=E\{h(X_1)|X_0=x\},\\
    Bh(y)\triangleq &\int h(y)\frac{T(x,y)\pi(x)}{\pi(y)}dy=E\{h(X_0)|X_1=y\}.
\end{align*}
The operator $F$ can be considered as the continuous state generalization of the transition matrix $T$ for finite state Markov chain (with $Tv$ as the operation on vector space).  Similarly, the operator $B$ can be considered as the generalization of the transpose of $T$. With this definition, we can see that
\begin{align*}
    E\{h(X_t)|X_0=x\}=F^th(x)\ \text{ and }\ E\{h(X_0)|X_t=y\}=B^th(y).
\end{align*}

Define the norm of an operator $F$ to be the operator norm induced by the $L_0^2(\pi)$ norm defined in \eqref{eq:inner}. By iterative variance formula
\begin{align*}
    \text{var}\{h(X_1)\}=E[ \text{var}\{h(X_1|X_0)\}]+\text{var}[E\{h(X_1)|X_0\}]\leq \text{var}[E\{h(X_1)|X_0\}],
\end{align*}
 and hence the norm of $F$ and $B$ are both less than or equal to one.
By the Markov property, $F$ and $B$ are adjoint to each other
\begin{align*}
    \langle Fh,g\rangle=\langle h,Bg\rangle.
\end{align*}
Since nonzero constant functions are excluded from $L_0^2(\pi)$, the spectral radius $r_F$ of $F$ is strictly less than one under mild conditions \citep{Liu1995}, which is defined by
\begin{align*}
    r_F=\lim_{t\rightarrow\infty}||F^t||^{\frac{1}{t}}<1.
\end{align*}
Lemma 12.6.3 in \cite{Jun2001} provides a Markov chain convergence bound in terms of the operator norm of $F^t$,
\begin{align}\label{eq:nor}
    ||T^t\circ p_0-\pi||_{L^2(\pi)}\leq ||F^t||\cdot||p_0-\pi||_{L^2(\pi)},
\end{align}
where $||p-\pi||^2_{L^2(\pi)}=\int (p(z)-\pi(z))^2/\pi(z)dz$ and $||p-\pi||_1\leq||p-\pi||_{L^2(\pi)}$ holds for any probability measure $p$.
Theorem 2.1 in \cite{roberts1997} shows that if \eqref{eq:nor} is true for a time reversible Markov chain with transition kernel $T$, then the chain is geometric ergodic with that same rate function, i.e. there exists a potential function $V:\mathcal{X}\to [1,\infty]$, such that
\begin{align*}
    ||T^t(x,\cdot)-\pi(\cdot)||_1\leq V(x)||F^t||,\quad x\in\mathds{X}.
\end{align*}
Therefore,
\eqref{eq:nor} implies a geometric convergence in $L_1$ norm with rate function $r(t)=||F^t||\sim r_F^t$. On the other side, by Lemma 12.6.4 in \cite{Jun2001},
\begin{align}\label{eq:corr}
    \sup_{g,h\in L^2(\pi)}\text{corr}\{g(X_0),h(X_t)\}
    =\sup_{||g||=1,||h||=1}\langle F^th,g\rangle=||F^t||.
\end{align}
This suggests the maximal autocorrelation function is of the same decay rate as the rate function $r(t)$. In practice, for multidimensional process $X_t=(X_{1,t},\ldots,X_{p,t})$, the above quantity can often be well approximated by
\begin{align*}
    \max_{j=1,\ldots,p}|\text{corr}\{X_{j,0},X_{j,t}\}|.
\end{align*}
Therefore, the maximal sample autocorrelation function provides a quantitative description of the mixing rate of the Markov chain, which provides the rationale for our choice of $m_t$ in section \ref{se:mt}.

If the Markov chain is reversible, then $F=B$ and hence $F$ is self-adjoint. Under the further assumption that $F$ is compact, $||F^t||=|\lambda_1|^t$, where $|\lambda_1|\geq|\lambda_2|\geq\cdots$ are the discrete eigenvalues of $F$. Therefore the rate function would be $r(t)=|\lambda_1|^t$.
For any $h(x)\in L_0^2(\pi)$, define the autocorrelation function as
\begin{align*}
    f(t)=\text{corr}\{h(X_t),h(X_0)\}, t\geq1.
\end{align*}
Let $\alpha_1(x),\alpha_2(x),\ldots$ be the corresponding eigenfunctions. Then as long as $\langle h,\alpha_1\rangle\neq 0$, we have
\begin{align*}
    \lim_{t\rightarrow\infty}\{|f(t)|\}^{1/t}=|\lambda_1|,
\end{align*}
which implies that the autocorrelation function and the rate function are very similar, i.e.
\begin{align}\label{eq:rate}
f(t)\sim r(t)\sim |\lambda_1|^t.
\end{align}

\section{Simulation with Finite Gaussian Mixtures}\label{se:mixm}
The mixing rate of Gibbs samplers are notoriously slow for mixture models \citep{jasra2005}. As an illustrative example, we consider the Bayesian Gaussian mixture model of \cite{richardson1997}, which is also considered by \cite{Moral2006} as a benchmark to test their method. Observations $y_1,\ldots,y_n$ are i.i.d. distributed as
\begin{align}\label{eq:mix}
    [y_i\ |\ \mu_{1:k},\lambda_{1:k},w_{1:k}]\ \sim\ \sum_{j=1}^kw_j N(\mu_j,\lambda_j^{-1}),
\end{align}
where $\lambda_{1:k}$ and $\lambda_{1:k}$ are the means and inverse variances of $k$ Gaussian components respectively, and $w_{1:k}$ are the mixing weights satisfying the constraint $\sum_{j=1}^kw_j=1$. The priors for the parameters of each component $j=1,\ldots,k$ are taken to be exchangeable as $\mu_j\sim N(\zeta,\kappa^{-1})$, $\lambda_j\sim Ga(\alpha,\beta)$, $w_{1:k}\sim Diri(\delta)$, where $Ga(\alpha,\beta)$ is the gamma distribution with shape $\alpha$ and rate $\beta$ and $Diri(\delta)$ is the Dirichlet distribution with number of categories $k$ and concentration parameter $\delta$. To enable a Gibbs sampler for the above model, we introduce for each observation $i=1,\ldots,n$ a latent class indicator $z_i$ such that
\begin{align*}
    &[y_i\ |\ z_i=j,\mu_{1:k},\lambda_{1:k},w_{1:k}] \sim  N(\mu_j,\lambda_j^{-1}),\\
    &P(z_i=j|w_{1:k})\ \propto \ w_j.
\end{align*}
Then by marginalizing out $z_i$'s, we can recover \eqref{eq:mix}.
With the above exchangeable prior, the joint posterior distribution $P(\mu_{1:k}|y_1,\ldots,y_n)$ of the $k$ component means $\mu_{1:k}$ has $k!$ modes and the marginal posterior for each $\mu_j$, $j=1,\ldots,k$ is the same as a mixture of $k$ components. Therefore, we can diagnose the performances of various samplers by comparing the marginal posteriors of $\mu_1,\ldots,\mu_k$.  Standard MCMC algorithms tend to get stuck for long intervals in certain local modes, and even a very long run cannot equally explore all these modes \citep{jasra2005}.

In this simulation, we generate the data with $n=100$ samples and choose the true model as $k=4$, $\mu_{1:4}=(-3, 0, 3, 6)$, $\lambda_{1:4}=(0.55^{-2}, 0.55^{-2}, 0.55^{-2}, 0.55^{-2})$ and $w_{1:4}=(0.25,0.25,0.25,0.25)$, which has the same settings as in \cite{jasra2005}
and \cite{Moral2006}. The hyperparameters for the priors are: $\zeta=0$, $\kappa=0.01$, $\alpha=1$, $\beta=2$ and $\delta=1$. We consider a batch setup with batch size (BS) $1, 2, 4, 6, 8$ and $10$, which means that data arrive in batches of size BS. As a result, the algorithms operate $T=\lceil 100/BS\rceil=100,50,25,17,13$ steps, where $\lceil x\rceil$ stands for the smallest integer no less than $x$.

In SMCMC, the dimension of the parameter $\theta^{(t)}=(\mu_{1:k},\lambda_{1:k},w_{1:k},z_{1:n_t})$ at time $t$ is increasing when the latent class indicators $z_{1:t}$ are included, where $n_t=0$ for $t=0$ or $100-BS\cdot(T-t)$ for $t=1,\ldots,T$ is the data size at time $t$. We choose the transition kernel $T_t$ to correspond to that for the Gibbs sampler. The jumping kernel $J_t$ is the conditional distribution for the additional latent indicators of $y_{(n_{t-1}+1):n_t}$ given $\theta^{(t)}$ and $y_{1:n_t}$. Note that $z_i$ are conditionally independent of $z_j$ for $i\neq j, i,j\leq n_t$ given $(\mu_{1:k},\lambda_{1:k},w_{1:k},y_{1:n_t})$.

We compare SMCMC with two competitors. The first algorithm is the sequential Monte Carlo (SMC) sampler in \cite{Moral2006}, which avoids data augmentation and works directly with the posterior of $(\mu_{1:k},\lambda_{1:k},w_{1:k})$ using MH kernels. The second algorithm is the parallel Gibbs sampler \citep{richardson1997} running on the full data $y_1,\ldots,y_n$, with $L$ Gibbs samplers running in parallel, whose iterations $K_{BS}$ equal the total Gibbs steps $\sum_{t=1}^Tm_t$ in the SMCMC with batch size $BS$. The posterior distribution of each $\mu_j$ with $j=1,2,3,4$ is approximated by the empirical distribution of the $L$ samples at $K_{BS}$th iteration in parallel. To demonstrate the annealing effect of SMCMC, the initial distributions of the $L$ chains for both SMCMC and MCMC (parallel Gibbs) are centered at $(-3,0,3,6)$.  As a result, if no pair of labels are switched, the posterior draws will be stuck around the local mode centered at $(-3, 0, 3, 6)$, which is one of the $4!=24$ local modes.

To compare the three algorithms, we calculate the averages of sorted estimated means across 10 trials under each setting as shown in Table~\ref{table:1}. More specifically, we sort the estimated posterior means of $\mu_{1:4}$ in increasing order for each run and then average the $4$ sorted estimates over 10 replicates.
A good algorithm is expected to provide similar posterior means of $\mu_{1:4}$, which is approximately 1.5 in our case. The purpose for sorting the estimated means is to prevent the differences in the estimated posterior means being washed away from averaging across 10 replicates.

As can be seen from Table~\ref{table:1}, SMCMC outperforms both SMC and MCMC under each setting and has satisfactory performance even when the batch size is 6, i.e. the number of time steps $T$ is 17. Moreover, the performance of SMCMC appears stable as the batch size grows from $1$ to $6$, and become worse when the batch size increases to $8$ and $10$. A similar phenomenon is observed for SMC, with performance starting to deteriorate at batch size $6$. MCMC has slightly worse performance with batch size 1 than SMCMC. However, its performance rapidly becomes bad as the number of iterations decreases. The comparison between SMCMC and MCMC illustrates substantial gains due to annealing for our method.

Figure~\ref{fig:1} displays some summaries for SMCMC with batch size 1. The left plot shows the number of Gibbs iteration $m_t$ versus time $t$ (which is equal to the sample size at time t). $m_t$ increases nearly at an exponential rate, which indicates the slow mixing rate of the Gibbs sampler used to construct the transition kernels $T_t$. As a by product of SMCMC, we can assess the convergence rate of the sampler used to construct $T_t$ as a function of the sample size. The automatic mixing diagnostics procedure guarantees the convergence of the approximated posterior as $t\to\infty$. The right panel shows the ``traceplot" for $\mu_{1:k}$ for one Markov chain among the $L$ chains. This is not the usual traceplot since we selected the last samples of $\mu_{1:k}$ at each time $t$, where $\mu_{1:k}$ is approximately distributed according to a time changing posterior $\pi_t$. This ``traceplot" suggests satisfactory mixing of $\mu_{1:k}$, i.e. frequent moves between the modes.

\begin{center}
\begin{table}
  \caption{Averages of sorted estimated means in mixture model by three approaches. We ran each algorithm 10 times with 1000 Markov chains or particles. We sorted the estimated means in increasing order for each run and then averaged the sorted estimates over 10 replicates. The last column reports the sample standard deviations of the first 4 numbers displayed. In the parenthesis following MCMC are the number of iterations it runs, which is equal to the average iteration the corresponding SMCMC runs across 10 replicates.}
  \centering
  \fbox{%
\begin{tabular}{lC{1.7cm}C{1.7cm}C{1.7cm}C{1.7cm}C{1.7cm}}
  \multirow{2}{*}{Algorithm description} & \multicolumn{4}{c}{Averages of sorted estimated component means} & standard\\
    & $\mu_1$ & $\mu_2$ & $\mu_3$ & $\mu_4$ & deviation \\
    \hline
   SMCMC (batch size 1) & 1.38 & 1.50 & 1.57 & 1.67 & 0.12\\
   SMC (batch size 1) & 1.13 & 1.37 & 1.60 & 1.97 & 0.36\\
   MCMC (8621 iterations) & 1.31 & 1.42 & 1.56  & 1.77 & 0.20\\
   \hline
   SMCMC (batch size 2) & 1.40 & 1.50 & 1.56 & 1.66 & 0.11\\
   SMC (batch size 2) & 1.22 & 1.46 & 1.75 & 1.99 & 0.34\\
   MCMC (4435 iterations) & 0.91 & 1.12 & 1.30  & 2.69 & 0.81\\
   \hline
   SMCMC (batch size 4) & 1.42 & 1.50 & 1.54 & 1.64 & 0.09\\
   SMC (batch size 4) & 1.57 & 1.84 & 2.01 & 2.32 & 0.31\\
   MCMC (2367 iterations) & 0.23 & 0.71 & 1.20  & 3.34 & 1.37\\
   \hline
   SMCMC (batch size 6) & 1.36 & 1.48 & 1.59 & 1.65 & 0.13\\
   SMC (batch size 6) & 1.31 & 1.63 & 1.93 & 2.35 & 0.44\\
   MCMC (1657 iterations) & -0.23 & 0.53 & 1.32  & 4.45 & 2.05\\
      \hline
   SMCMC (batch size 8) & 1.35 & 1.45 & 1.54 & 1.73 & 0.16 \\
   SMC (batch size 8) & 1.43 & 1.69 & 1.99 & 2.35 & 0.40\\
   MCMC (1390 iterations) & -0.50 & 0.53 & 1.36  & 4.68 & 2.24\\
      \hline
   SMCMC (batch size 10) & 1.19 & 1.32 & 1.57 & 2.04 & 0.37\\
   SMC (batch size 10) & 1.36 & 1.69 & 1.98 & 2.38 & 0.43\\
   MCMC (1069 iterations) & -1.00 & 0.38 & 1.60  & 5.11 & 2.62\\
\end{tabular}}
  \label{table:1}
\end{table}
\end{center}

\begin{figure}[htp]
\centering
\begin{tabular}{cc}
    \includegraphics[width=2.7in]{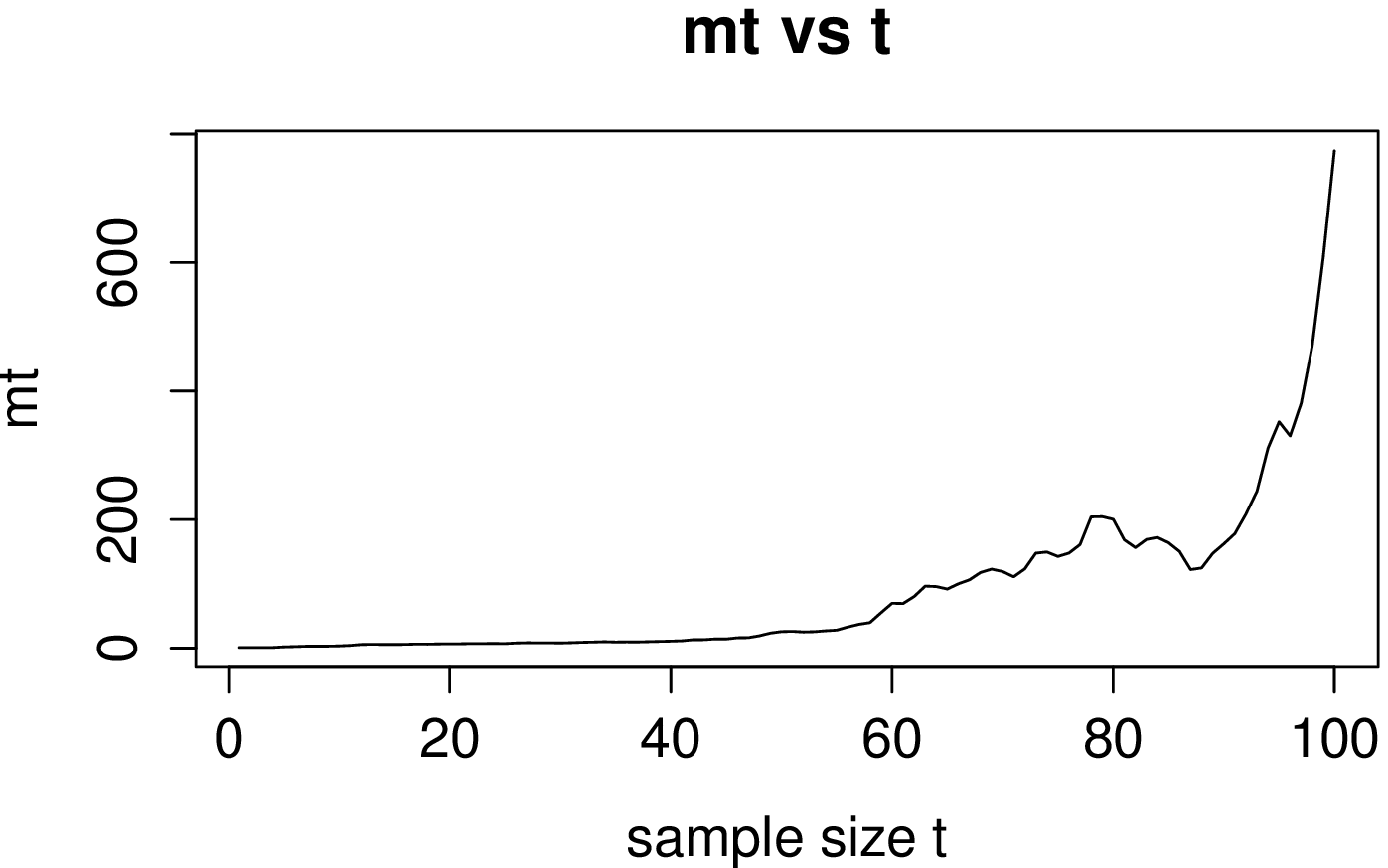}
    &
    \includegraphics[width=2.7in]{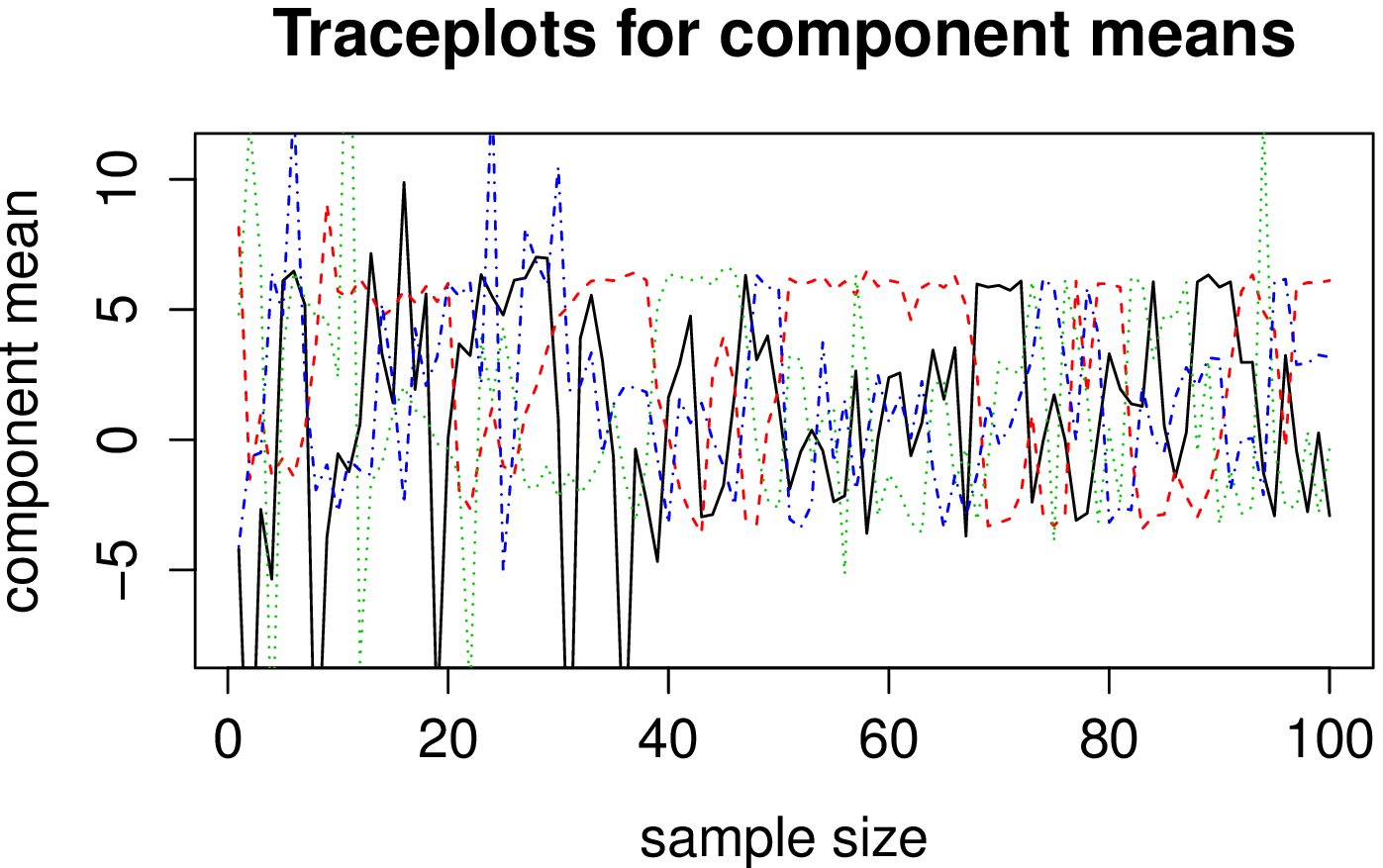}
\end{tabular}
\caption{Summaries of SMCMC with batch size 1. The left panel displays the plot of the number of Gibbs iterations $m_t$ versus time $t$ (which is equal to the sample size at time t). The right panel displays the last samples of $\mu_{1:k}$ at each time $t$ in one of $L$ Markov chains.}
\label{fig:1}
\end{figure}

\vspace{-1cm}
\section{Sequential Bayesian Estimation for Heart Disease Data}\label{se:npr}
In the following we apply SMCMC to a sequential, growing dimension nonparametric problem. We consider nonparametric probit regression with a Gaussian process (GP) prior. Let $y_1,y_2,\ldots$ be a sequence of binary responses and $x_1,x_2,\ldots$ the $p$ dimensional covariates. The model assumes
\begin{align*}
    P(y_i=1)=\Phi(f(x_i)),
\end{align*}
where $\Phi$ is the cdf of the standard normal distribution and $f$ is a $d$-variate nonlinear function.
We choose a GP as a prior, $f\sim GP(\kappa,K)$, with mean function $\kappa:\bbR^p\to\bbR$ and covariance function $K:\bbR^d\times\bbR^d\to\bbR$. We consider the squared exponential kernel $K_a(x,x')=\sigma^2\exp\{-a^2||x-x'||^2\}$ with a powered gamma prior on the inverse bandwidth, which leads to an adaptive posterior convergence rate \citep{Van2009}.

The computation of the nonparametric probit model can be simplified by introducing latent variables $z_i$ such that
\begin{equation}\label{eq:pz}
\begin{aligned}
    &P(y_i=1)=I(z_i>0),\\
    &z_i=f(x_i)+\epsilon_i,\epsilon_i\sim N(0,1).
\end{aligned}
\end{equation}
The model has simple full conditionals so that a Gibbs sampler can be used to sample the $z_i$'s and $F_t=\{f(x_1),\ldots,f(x_t)\}$.

To alleviate the $O(n^3)$ computational burden of calculating inverses and determinants of $n\times n$ covariance matrices, we use a discrete prior to approximate the powered gamma prior for $a$ and pre-compute those inverses and determinants over the pre-specified grid. We combine the sequential MCMC with the following off-line sequential covariance matrix updating.

Let $a_1,\ldots,a_H$ denote a grid of possible inverse bandwidths. For example, $a_h$ can be chosen as the $\frac{h-1}{H}$th quantile of the powered gamma prior and the discrete prior as the uniform distribution over $a_1,\ldots,a_H$. Let $C_h(x,x')=\exp\{-a_h^2||x-x'||^2\}$ and $K_{a_h}=\sigma^2C_h$. We use the notation $C(A,B)$ to denote the matrix $(c(a_i,b_j))_{p,q}$ for a function $C:\bbR^d\times\bbR^d\to\bbR$ and matrices $A\in\bbR^{p\times d}$, $B\in\bbR^{q\times d}$. Let $X_t=(x_1^T,\ldots,x_t^T)^T\in\bbR^{t\times d}$, $Y_t=(y_1,\ldots,y_t)$ and $Z_t=(z_1,\ldots,z_t)$ be the design matrix, response vector and latent variable vector at time $t$. At time $t$, for each $h=1,\ldots,H$, we update the lower triangular matrix $L_h^{(t)}$ and $(L_h^{(t)})^{-1}$ in the Cholesky decomposition $C_{h}^{(t)}=L_h^{(t)}(L_h^{(t)})^T$ of the $t\times t$ correlation matrix $C_{h}^{(t)}=C_h(X_t,X_t)$. The reason is two-fold: 1. inverse and determinant can be efficiently calculated based on $L_h^{(t)}$ and $(L_h^{(t)})^{-1}$; 2. due to the uniqueness of Cholesky decomposition, $L_h^{(t+1)}$ and $(L_h^{(t+1)})^{-1}$ can be simply updated by adding $(t+1)$th row and column to $L_h^{(t)}$ and $(L_h^{(t)})^{-1}$. More precisely, if $L_h^{(t+1)}$ and $(L_h^{(t+1)})^{-1}$ are written in block forms as
\begin{align*}
    L_h^{(t+1)}=\left(
      \begin{array}{cc}
        L_h^{(t)} & 0 \\
        B_h^{(t+1)} & d_h^{(t+1)} \\
      \end{array}
    \right)\text{ and }
    (L_h^{(t+1)})^{-1}=\left(
      \begin{array}{cc}
        (L_h^{(t)})^{-1} & 0 \\
        E_h^{(t+1)} & g_h^{(t+1)} \\
      \end{array}
    \right),
\end{align*}
where $B_h^{(t+1)}$ and $E_h^{(t+1)}$ are $t$-dimensional row vectors and $d_h^{(t+1)}$ and $g_h^{(t+1)}$ are scalars, then we have the following recursive updating formulas: for $h=1,\ldots,H$,
\begin{align*}
    d_h^{(t+1)}=&\big\{C_h(x_{t+1},x_{t+1})-C_h(x_{t+1},X_t)(L_h^{(t)})^{-T}(L_h^{(t)})^{-1}
    C_h(X_t ,x_{t+1})\big\}^{1/2},\\
    B_h^{(t+1)}=&C_h(x_{t+1},X_t)(L_h^{(t)})^{-1},\\
    g_h^{(t+1)}=&(d_h^{(t+1)})^{-1},\\  E_h^{(t+1)}=&-g_h^{(t+1)}C_h(x_{t+1},X_t)(L_h^{(t)})^{-T}(L_h^{(t)})^{-1},
\end{align*}
where for a matrix $A$, $A^{-T}$ is a shorthand for the transpose of $A^{-1}$. The computation complexity of the above updating procedure is $O(t^2)$.

As $t$ increases to $t+1$, the additional component $\eta_{t+1}$ is $(f(x_{t+1}),z_{t+1})$. Therefore, in the jumping step of the sequential updating, we repeat drawing $f(x_{t+1})$ and $z_{t+1}$ from their full conditionals in turn for $r$ times. In our algorithm, we simply choose $r=1$ as the results do not change much with a large $r$.
In the transition step of the sequential updating, each full conditional is recognizable under the latent variable representation \eqref{eq:pz} and we can run a Gibbs sampler at each time $t$. Predicting draws $f(x')$ on new covariates $x'$ can be obtained based on posterior samples of $F_t$.

Note that the computational complexity for the off-line updating at time $t$ is $O(t^2)$. Therefore the total complexity due to calculating matrix inversions and determinants is $O(\sum_{t=1}^nt^2)=O(n^3)$, which is the same as the corresponding calculations in the MCMC with all data. However, the proposed algorithm distributes the computation over time, allowing real-time monitoring and extracting of current information.

To illustrate the above approach, we use the south African heart disease data \citep{rousseauw1983,Hastie1987} to study the effects of obesity and age on the probability of suffering from hypertension. The data contains $n=462$ observations on 10 variables, including systolic blood pressure (sbp), obesity and age. A patient is classified as hypertensive if the systolic blood pressure is higher than 139 mmHg. We use $I(\text{sbp}>139)$ as a binary response with obesity and age as a two-dimensional covariate $x$.

Fig.~\ref{fig:2} demonstrates the relationship between the number of iterations $m_t$ and the sample size $t$. As can be seen, $m_t$ keeps fluctuating between 150-200 as $t$ becomes greater than 100, indicating that contrary to the mixture model example, the mixing rate of the above Markov chain designed for the nonparametric probit regression is robust to the sample size. The total number of iterations $\sum_{t=1}^nm_t$ is about 80k. However, the computation complexity of each SMCMC chain is much less than a 80k iterations full data MCMC since many iterations of SMCMC run with smaller sample sizes. In addition, we can reduce the iterations needed by increasing block sizes.

Fig.~\ref{fig:3} shows the fitted probabilities of hypertension as a function of obesity and age at $t=150, 250, 350, 462$. With a relatively small sample size, the bandwidth $a^{-1}$ tends to be small and the fitted probability contours are wiggly. As the sample size $t$ increases, the bandwidth grows. As a result, contours begin to capture some global features and are less affected by local fluctuations. In addition, at large time point $t=350$, the posterior changes little as the sample size further grows to $t=462$.  As expected, the probability of hypertension tends to be high when both obesity index and age are high. The gradient of the probability $P(\text{sbp}>139|\text{obesity,age})$ as a function of obesity and age tends to be towards the $45$-degree direction. The results in Fig.~\ref{fig:3} are indistinguishable from those obtained running a long MCMC at each time, which are omitted here.

\section{Discussions}
In this paper, we proposed a sequential MCMC algorithm to sample from a sequence of probability distributions. Supporting theory is developed and simulations demonstrate the potential power of this method. The performance of SMCMC is closely related to the mixing behavior of the transition kernel $T_t$ as $t\to \infty$. If $T_t$ tends to have poor mixing as $t$ increases, then updating the ensemble $\Theta_t$ every time a new data point arrives can lead to increasing computational burden over time.
To alleviate this burden, we have three potential strategies. First, we can make the updating of $\Theta_t$ less frequent as $t$ grows, i.e. updating $\Theta_t$ only at time $\{t_k:k=1,\ldots\}$ with $t_k\to \infty$ as $k\to \infty$ and $t_k-t_{k-1}\to \infty$, as long as $||\pi_{t_k}-\pi_{t_{k-1}}||_1\to 0$.
Second, we can let the $\epsilon$ in Algorithm \ref{al:1} decrease in $t$ so that the upper bound in Theorem \ref{thm:main} still converges to zero. Third, we can develop `forgetting' algorithms that only use the data within a window but still guarantee the convergence up to approximate error. The first two strategies may also be developed in an adaptive/dynamic manner, where the next step size $t_{k+1}-t_k$ or decay rate $\epsilon_{k+1}$ is optimized based on some criterion by using the past data and information.

\section*{Acknowledgments}
The authors thank Natesh Pillai for helpful comments on a draft.

\begin{figure}[htp]
\centering
    \includegraphics[width=3.5in]{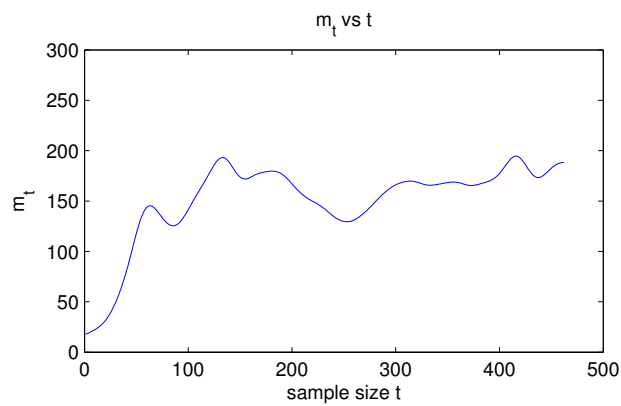}
\caption{The iterations $m_t$ at time $t$ versus the sample size $t$ is displayed. $m_t$ has been smoothed with window width equal to 10.}
\label{fig:2}
\end{figure}
\begin{figure}[htp]
\centering
    \includegraphics[width=6in]{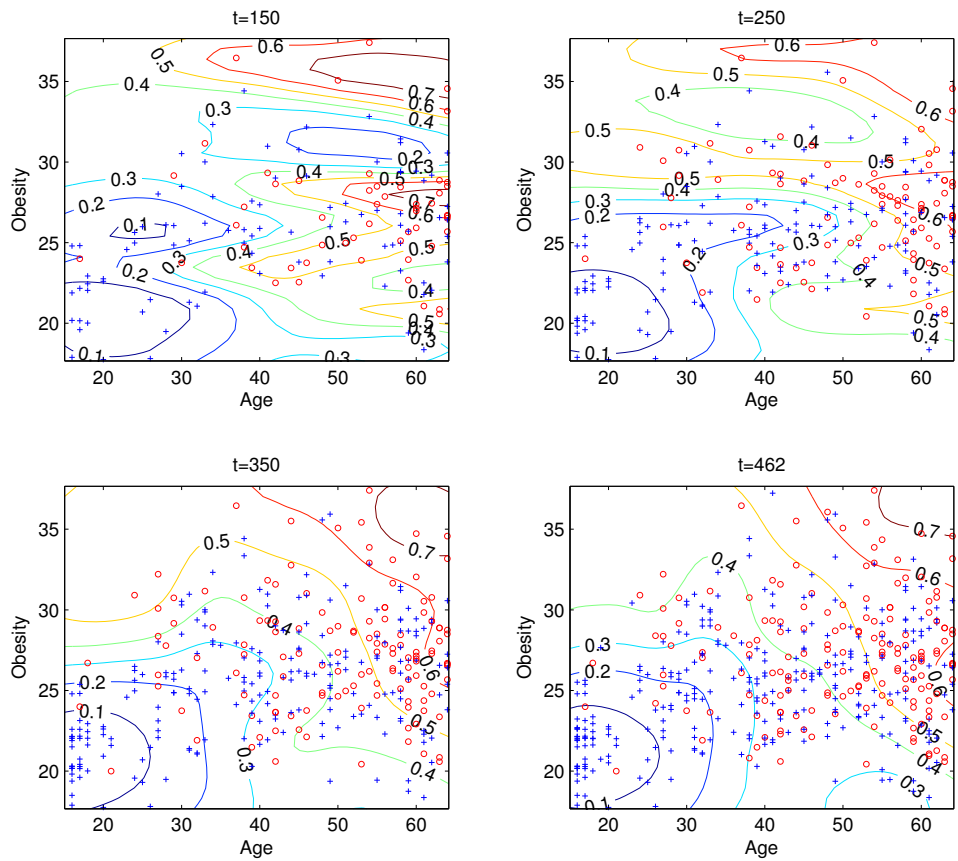}
\caption{The fitted hypertension probability contours at $t=150,250,350,462$. The circles correspond to hypertensive patients and plus signs correspond to normal blood pressure people.}
\label{fig:3}
\end{figure}

\bibliography{draft}
\end{document}